%% file: main.tex
\documentclass[11pt]{article}
\usepackage[letterpaper, portrait, margin=1in]{geometry}
\usepackage{amsmath}
\usepackage{amsthm} 
\usepackage{amssymb}
\usepackage{graphicx}
\usepackage{caption}
\usepackage{float} 
\usepackage{mathtools}
\usepackage{hyperref}
\usepackage{xcolor}
\usepackage{enumitem}
\usepackage{tikz}
\usepackage{comment}
\usepackage{authblk}
\usepackage[labelformat=simple]{subcaption}

\newtheorem{theorem}{Theorem}[section]

\newtheorem{corollary}{Corollary}[theorem]
\newtheorem{lemma}[theorem]{Lemma}
\newtheorem{proposition}[theorem]{Proposition}
\newtheorem{remark}{Remark}[section]

\title{On the Number of Vertices in Hyperplane Cross-Sections \\ of Convex Polytopes}
\author[1]{Jes\'us A. De Loera}
\author[1]{Gyivan Lopez-Campos \footnote{Corresponding author}}
\author[1]{Antonio J. Torres}
\affil[1]{Department of Mathematics, University of California, Davis.}
\date{}                     
\date{} 

\begin{document}
\maketitle



\maketitle


\begin{abstract}
We study the combinatorics of the slices or sections of a convex polytope by affine hyperplanes. 
We present results on two key problems: First, we provide tight bounds on the maximum number of vertices attainable by a hyperplane slice of a $d$-polytope 
(a sort of upper bound theorem) and discuss a new algorithm to find all combinatorially distinct sections. Second, we investigate the sequence of numbers of vertices produced by the different slices over all possible hyperplanes and analyze the gaps that arise 
in that sequence. We study these sequences for three-dimensional polytopes, 
cyclic polytopes, and hypercubes. Some of our results were obtained with the 
help of large computational experiments, and we report new data generated for hypercubes. 
\end{abstract}

\section{Introduction}  
The study of the hyperplane slices or sections of a convex polytope dates back to at least the early 1800's with the work of Laplace; who studied formulas for the volume of slices of cubes \cite{laplace1835oeuvres}. Since then, the volumes and other analytics properties of slices or sections of convex polytopes have been extensively explored in the literature due to its significance to analysis and probability (see, e.g., \cite{ball2006volumes, koldobsky2023harmonic,GiannopoulosKoldobskyZvavitch2023, KlaMil22:SlicingProblem, klartag2024affirmativeresolutionbourgainsslicing,nayar+tkocz-2022extremalsecproj,pournin-scubesections} and the references therein). 
Only more recently has there been an effort to understand the combinatorial properties of convex polytope slices (see \cite{barany2023same,brandenburg2025best, fukuda1997sections, Khovanskii-slices2006, Lawrence79,o1971hyperplane,PadrolPfeifle2015} and the 
references therein). The purpose of this paper is to count vertices within slices.

In what follows all polytopes are convex, thus we often omit the convex adjective.
For a polytope $P$ we denote by $V(P)$ and $E(P)$ the sets of vertices and edges of $P$, respectively, and let $\mathcal{H}_P$ be the set of hyperplanes in $\mathbb{R}^d$ whose intersection with $P$ is non-empty. Formally, if $P$ is a polytope of dimension $d$ and $H\in \mathcal{H}_P$, then $H \cap P$  is called a \emph{slice} of $P$. Note that each slice of $P$ is an $n$-dimensional convex polytope with $n < d$. Note also a point $v$ is a vertex of $H \cap P$ if and only if either (i) $v$ is a vertex of $P$ lying in $H$, or (ii) $v$ is the unique interior point where $H$ intersects an edge $e$ of $P$.

For a positive integer $n$, we let $[n]=1,2,\dots,n$.
In this paper, we are devoted to studying the values of the counting function $cv_P:\mathcal{H}_P\rightarrow\mathbb{N}$, with $cv_P(H)=|V(H \cap P)|$. The authors of \cite{brandenburg2025best} provided the first general algorithm to compute the maximum of $cv_P$ (for a fixed $d$-polytope $P$), which we denote by $\nu(P)$. We address two fundamental problems:

\begin{itemize}
    \item Given a $d$-polytope $P$, what are tight bounds to its maximum $\nu(P)$? Computationally, there are prior algorithms to find optimal slices (those that reach $\nu(P)$) for fixed dimension, but the computational complexity of finding them, for an specific input polytope $P$ of non-fixed dimension remains unknown (see \cite{brandenburg2025best}). More generally, we aim to know the maximum values of $\nu(P)$ among all $d$-polytopes $P$ with $n$ vertices. Prior work on bounding the entries of the $f$-vector of slices of $P$ can be found in \cite{Khovanskii-slices2006}.

    \item Given a $d$-polytope $P$, the \emph{vertex slice sequence} of $P$, denoted by $VSS(P)$, is the sequence generated by the values in the image 
    of $cv_P$ (i.e., $cv_P(\mathcal{H}_P)$). A value $m\in [\nu(P)]$ that is missing from $VSS(P)$ is called a \emph{gap} of $P$.
    The questions we are interested in are: What is the structure of $VSS(P)$? When does a given polytope $P$ have gaps? For a fixed dimension, what is the maximum gap over all the $d$-polytopes? We focus on these questions for $3$-polytopes, cyclic polytopes and hypercubes. 
    
\end{itemize}

\section*{Our Contributions:} 

In Section \ref{Sec:Cylic},  we discuss the value of $\nu(P)$, for a fixed $d$-polytope $P$. We also show that the cyclic polytope maximizes $\nu(P)$, among all the $d$-polytopes. This result represents partial progress on a broader question (Open problem 7.7 in \cite{Khovanskii-slices2006}) on maximizing $i$-dimensional faces of slices of polytopes.

\begin{theorem} \label{Them:max vertices}
Let ${\cal P}_{d,n}$ be the set of $d$-dimensional convex polytopes with $n$ vertices. Then
\[
\max_{P \in \mathcal{P}_{d,n}} \nu(P)
\]
is attained by the cyclic polytope and it is equal to
$2(n-2)$ in dimension $d=3$ and it equals  $\lfloor n/2 \rfloor \lceil n/2 \rceil$ for dimension $d>3$.
\end{theorem}

Furthermore, for the proof of Theorem \ref{Them:max vertices} we successfully describe the entire VSS for the cyclic polytope.\\

Regarding how to compute all slices of a polytope, in Section \ref{sec:Posets} we show a new algorithm to enumerate all combinatorially different slices  using a partially order set (poset). Although this algorithm is less efficient than the one given in \cite{brandenburg2025best}, it is a useful framework for proving theoretical results, as we demonstrate in Section \ref{Section:Hypercube} for the hypercube. 


\begin{theorem}\label{Thm:antichain}
Let $P$ be a $d$-polytope. For each direction $u \in \mathbb{S}^{d-1}$, there exists a partial order on the vertices and the edges of $P$ that are not orthogonal to $u$, such that for any hyperplane $H \in \mathcal{H}_P$ with normal vector $u$, the set of vertices and edges intersected by $H$ forms a maximal antichain in this poset.
\end{theorem}

In Section \ref{Sec:GapsR3}, we prove the following theorem concerning the vertex sequences of slices of $3$-polytopes.  

\begin{theorem}\label{Thm:GAPSR3}
    There are infinitely many $3$-polytopes, which are only $3$-connected with at least one gap in their sequence of the number of vertices of slices. 
\end{theorem}

Let the hypercube, denoted by $Q_d$, be the convex hull of the $2^d$ distinct points in $\mathbb{R}^d$ where each coordinate is either $0$ or $1$. 
    While the volume of slices of $Q_d$ has received considerable attention (see, e.g., \cite{laplace1835oeuvres,ball2006volumes,GiannopoulosKoldobskyZvavitch2023,pournin-scubesections}), its combinatorics has been much less explored, with only a few papers addressing the topic (\cite{ChakerianLogothetti91,fukuda1997sections,Lawrence79,o1971hyperplane}). Similar work on $VSS(Q_d)$ has been conducted in \cite{fukuda1997sections,Lawrence79} for $d\leq5$. In Section \ref{Section:Hypercube},  we present two results that fully describe the VSS of hypercubes of dimension less than or equal to seven:

\begin{theorem}\label{Thm:FirstGAPS}
Let $d\geq 4$, and $I=\{2^{i}, \text{ for }0\leq i\leq d-1\}\cup\{d, 2d-2, 3d-5, 3d-4\}$. Then, $[4d-10]\setminus I$ is the set of the first gaps of the $d$-dimensional hypercube $Q_d$.
\end{theorem}

\begin{theorem} \label{Thm:gap_penultimo}
    For every even dimension $d>4$,  the number $(d/2)\binom{d}{d/2}-1$ is a gap in the vertex sequence of slices of the $d$-dimensional hypercube $Q_d$.
\end{theorem}

Finally, using additional computational effort, we provide the following table summarizing the VSS for hypercubes up to dimension seven.

\begin{corollary}\label{Corollary:TableUpto7}
        The sequences of the number of vertices on slices of the hypercubes of dimension up to 7 is summarized in the following table.
\end{corollary}

{\small 
\begin{table}[ht]
    \centering
    \begin{tabular}{| c|c |}
     \hline
     \textbf{Dimension} & \textbf{VSS}\\
      \hline 
        2 & $[2]$ \\
         \hline
        3 & $[6]$\\
         \hline
        4 &  $[12]\setminus\{3,5\}$\\
          \hline
        5 & $[30] \setminus \{3, 6, 7, 9\}$\\
           \hline
        6  &$[60] \setminus \{3, 5, 7, 9, 11, 12, 59\}$\\
            \hline
        7  &$[140] \setminus \{3, 5, 6, 9, 10, 11, 13, 14, 15, 18\}$\\
            \hline
    \end{tabular}

\caption{\centering Sequences of vertices from slices (VSS) of hypercubes up to dimension seven. }
    \label{tab:VSS}
\end{table}
}


\section{Maximum number of vertices in Slices and the Cyclic Polytope}\label{Sec:Cylic}


The \emph{cyclic polytope}, denoted by $C_d(n)$, is a convex polytope obtained as the convex hull of $n$ points $\gamma(t_1), \gamma(t_2), \dots, \gamma(t_n)$ on the moment curve $\gamma(t) = (t, t^2, t^3, \dots, t^d) \in \mathbb{R}^d$, where $t_1 < t_2 < \dots < t_n$ are distinct real numbers (See, e.g., \cite{grunbaum2003convex,ziegler2012lectures} for a thorough introduction to cyclic polytopes). In this section, we characterize the vertex sequences of 
the slices of cyclic polytopes. The following lemma plays a crucial role in identifying slices with the maximum number of vertices, as it limits the 
search to those slices that intersect only the edges of the polytope.

\begin{lemma}\label{lemma 1}
Let $P$ be a $d$-polytope, and let $H \in \mathcal{H}_P$ be a hyperplane that intersects $P$ in at least one of its vertices. Then there exists a hyperplane $H' \in \mathcal{H}_P$ that intersects $P$ in none of its vertices and satisfies $cv_P(H') \geq cv_P(H)$.
\end{lemma}

\begin{proof}
The hyperplane divides $\mathbb{R}^d$ into two open half-spaces, $H^+$ and $H^-$. If all the vertices of $P$ are contained in the closure of one of these half-spaces, say $H^+$, then $H$ is tangent to $P$, and all the vertices of $P$ contained in $H$ must have neighbors in $H^+$. Otherwise, all the vertices of $P$ contained in $H$ have neighbors in both half-spaces.

Consider the hyperplane $H' \subset H^+$ obtained by moving $H$ continuously in the direction of its normal vector just until $H'$ no longer contains any vertices of $P$, but still intersects $P$. We claim that the number of vertices in $H' \cap P$ is at least as large as in $H \cap P$.

Observe that $H'$ intersects the relative interiors of the edges of $P$ that are also intersected by $H$. Additionally, for any vertex $v \in V(P) \cap H$, there is a neighbor $a(v)$ contained in $H^+$. As $H$ moves to $H'$, the hyperplane intersects the edge connecting $v$ and $a(v)$ in its interior, ensuring a new point of intersection within $P$.

Thus, $cv_P(H') \geq cv_P(H)$, completing the proof.
\end{proof}

One of the most fundamental structural results in the theory of cyclic polytopes is Gale's Evenness Condition \cite{gale1963neighborly}, which completely characterizes their facet structure. The theorem states the following: 

\begin{theorem}[Gale's Evenness Condition]
Consider the cyclic polytope induced by the convex hull of the image of the real parameters $t_1 < t_2 < \dots < t_n$ under the map $\gamma$. A subset $S \subseteq \{t_1, \ldots, t_n\}$ of size $d$ forms a facet of $C_d(n)$ if and only if, for any $t_i < t_j$ not in $S$, the number of elements $t_k \in S$ such that $t_i < t_k < t_j$ is even.
\end{theorem}

\begin{corollary}\label{linea dual}
Let $T = \{t_1 < t_2 < \dots < t_n\} \subset \mathbb{R}$ and $P = \{p_1 < p_2 < \dots < p_d\} \subset \mathbb{R} \setminus T$. We define the open intervals:
\[
I_1 = (-\infty, p_1),\quad I_2 = (p_1, p_2),\quad \dots,\quad I_d = (p_{d-1}, p_d),\quad I_{d+1} = (p_d, \infty).
\]

Let $H = \mathrm{aff}(\gamma(p_1), \dots, \gamma(p_d))$. Then, for every $t \in T$, if $t \in I_i$ with $i$ odd, we have $\gamma(t) \in H^+$; whereas if $i$ is even, then $\gamma(t) \in H^-$.
\end{corollary}

\begin{proof}
Consider the set of points $T \cup P \subset \mathbb{R}$ and the cyclic polytope $\mathrm{conv}(\gamma(T \cup P))$. If $\mathrm{conv}(\gamma(P))$ is a facet of this polytope, then all points of $\gamma(T)$ must lie on one side of the hyperplane $H = \mathrm{aff}(\gamma(P))$. Moreover, by Gale's evenness condition, it follows that either no interval with odd index contains points of $T$, or no interval with even index does.

On the other hand, if $\mathrm{conv}(\gamma(P))$ is not a facet of $\mathrm{conv}(\gamma(T \cup P))$, then there are points of $\gamma(T)$ on both sides of the hyperplane $H$. Consider the set $\mathrm{conv}(\gamma((T \cup P) \cap (\bigcup_{i\ \text{odd}} I_i)))$. By Gale’s evenness condition, $\mathrm{conv}(\gamma(P))$ is one of its facets, which implies that the points $\gamma(T \cap (\bigcup_{i\ \text{odd}} I_i))$ lie entirely in one half-space, say $H^+$. Similarly, the points $\gamma(T \cap (\bigcup_{i\ \text{even}} I_i))$ must lie in the other half-space, that is, $H^-$.
\end{proof}

\begin{theorem}\label{cyclic polytope slices}
The vertex sequences of the slices of the cyclic polytope $C_d(n)$ for $d \geq 3$ are the following:
\[
\begin{cases} 
    [2(n-2)] & \text{if } d = 3, \\
    \left\{a \cdot b + i \mid a, b \in \{0, 1, \dots, n-1\}, \, i \in \{0, 1, \dots, d\}, \, a + b = n - i \right\} & \text{if } d > 3.
\end{cases}
\]   
\end{theorem}

\begin{proof}
For the $3$-dimensional cyclic polytope, by \textbf{Lemma \ref{lemma 1}}, to find a plane that generates a slice of $C_3(n)$ with the maximum number of vertices, it is enough to consider planes that intersect only the edges of $C_3(n)$. We know that $C_3(n)$ has $3n - 6$ edges. If a plane $H$ intersects only edges, then it splits the $1$-skeleton into two connected subgraphs, one in $H^+$ and the other in $H^-$. Since a connected graph with $k$ vertices has at least $k - 1$ edges (its spanning tree), these two graphs must have at least $n - 2$ edges lying entirely within $H^+ \cup H^-$. Therefore, at most $(3n - 6) - (n - 2) = 2(n - 2)$ edges cross the plane $H$, which provides an upper bound for $\nu(C_3(n))$. 

We now show that every integer $m \in [2(n - 2)]$ can actually be attained as the number of vertices in a slice.

Slices with one, two, or three vertices correspond to tangent planes intersecting at a vertex, an edge, or a face, respectively. To produce the remaining values, consider the following three families of planes intersecting $C_3(n)$ only through edges:



\begin{align*}
\mathcal{A} &= \left\{ H = \mathrm{aff}\big(\gamma(t_a), \gamma(t_b), \gamma(t_c)\big) \;\middle|\; 
                t_a, t_b, t_c \notin \{t_i\}_{i=1}^n,\ 
                t_a < t_1,\ 
                t_2 < t_b < t_c < t_{n-1} \right\}, \\
\mathcal{B} &= \left\{ H = \mathrm{aff}\big(\gamma(t_a), \gamma(t_b), \gamma(t_c)\big) \;\middle|\; 
                t_a, t_b, t_c \notin \{t_i\}_{i=1}^n,\ 
                t_a < t_1,\ 
                t_2 < t_b < t_{n-1} < t_c < t_n \right\}, \\
\mathcal{C} &= \left\{ H = \mathrm{aff}\big(\gamma(t_a), \gamma(t_b), \gamma(t_c)\big) \;\middle|\; 
                t_a, t_b, t_c \notin \{t_i\}_{i=1}^n,\ 
                t_a < t_1 < t_b < t_2,\ 
                t_{n-1} < t_c < t_n \right\}.
\end{align*}

The relative positions of these hyperplanes with respect to the vertex set $V(C_3(n)) = \{\gamma(t_i)\}_{i=1}^n$ are characterized by:

\begin{itemize}
\item \textbf{Family $\mathcal{A}$}: 
  \begin{itemize}
  \item Contains between $t_b$ and $t_c$: $0 \leq k \leq n-4$ points from $\{t_3, \ldots, t_{n-2}\}$,
  \item Satisfies $\{\gamma(t_1), \gamma(t_2), \gamma(t_{n-1}), \gamma(t_n)\} \subset H^+$, and $|H^- \cap V(C_3(n))| = k$ with $1 \leq k \leq n-4$.
  \end{itemize}

\item \textbf{Family $\mathcal{B}$}: 
  \begin{itemize}
  \item Contains between $t_b$ and $t_c$: $0 \leq k \leq n-3$ points from $\{t_3, \ldots, t_{n-1}\}$,
  \item Satisfies $\{\gamma(t_1), \gamma(t_2), \gamma(t_n)\} \subset H^+$ and $\gamma(t_{n-1}) \in H^-$, and $|H^- \cap V(C_3(n))| = k$ with $1 \leq k \leq n-3$.
  \end{itemize}

\item \textbf{Family $\mathcal{C}$}: 
  \begin{itemize}
  \item Has fixed ordering relative to $\{t_i\}_{i=1}^n$,
  \item Satisfies $\{\gamma(t_1), \gamma(t_n)\} \subset H^+$ and $\{\gamma(t_2), \ldots, \gamma(t_{n-1})\} \subset H^-$.
  \end{itemize}
\end{itemize}

These classifications follow directly from applying Corollary \ref{linea dual}.

Then, for any plane $H \in \mathcal{A}$, the $k$ vertices lying in $H^-$ are adjacent to both $\gamma(t_1)$ and $\gamma(t_n)$, and two of these vertices are also adjacent to a vertex in $H^+$ different from $\gamma(t_1)$ and $\gamma(t_n)$. Hence, the number of edges crossing $H$ is $2k + 2$, with $1 \leq k \leq n - 4$.

Similarly, for $H \in \mathcal{B}$, the $k$ vertices in $H^-$ are adjacent to both $\gamma(t_1)$ and $\gamma(t_n)$, and one of them is adjacent to a vertex in $H^+$ different from $\gamma(t_1)$, $\gamma(t_2)$, and $\gamma(t_n)$. Therefore, the number of edges crossing $H$ is $2k + 1$, with $1 \leq k \leq n - 3$.

For $H \in \mathcal{C}$, the number of edges crossing the plane is exactly $2(n - 2)$, coming from all connections between $\gamma(t_1)$, $\gamma(t_n)$, and the remaining vertices.

Thus, for every integer $m \in [1, 2(n - 2)]$, there exists a hyperplane $H$ such that the number of edges in the slice \(H \cap C_3(n)\) is equal to \(m\).\\

Next, for dimension \(d > 3\), consider a hyperplane \(H\) intersecting the cyclic polytope \(C_d(n)\). Note that it always cuts the moment curve into $d$ points. Suppose \(H\) partitions the vertices of \(C_d(n)\) such that:
\begin{center}
   $\vert H^+ \cap V(C_d(n))\vert =a,$ \hskip .3cm $\vert H^- \cap V(C_d(n))\vert =b,$ \hskip .3cm  and \hskip .3cm   $\vert H \cap V(C_d(n))\vert =i$
\end{center}
where \(a, b \in \{0, 1, \dots, n-1\}\) and \(i \in \{0, 1, \dots, d\}\). By Corollary \ref{linea dual}, such hyperplane \(H\) can always be constructed as the affine span of a suitable selection of \(d\) points on the moment curve (not necessarily distinct from \(\gamma(t_1), \gamma(t_2), \ldots, \gamma(t_n)\)). Then, the number of vertices in the section \(H \cap C_d(n)\) is determined by the \(a \cdot b\) edges crossing \(H\), plus the \(i\) vertices lying in \(H\). Consequently, the  VSS is given by the following set of integers:
\[
\left\{a \cdot b + i \mid a, b \in \{0, 1, \dots, n-1\}, \, i \in \{0, 1, \dots, d\}, \, a + b = n - i \right\}. 
\]
\end{proof}

Unlike the three-dimensional case, the VSS of $C_d(n)$ for \(d > 3\) contains gaps. This is because if \(a + b < n - d\), then it follows that \(i > d\), which is not allowed. As a result, some values are not achievable. For example, if \(d = 4\) and \(n = 10\), the VSS of $C_4(10)$ is:
\[
\{1,\ 2,\ 3,\ 4,\ 9,\ 12,\ 13,\ 14,\ 15,\ 16,\ 17,\ 18,\ 19,\ 21,\ 24,\ 25\}.
\]
\vskip .3cm

The Upper Bound Theorem (UBT), a cornerstone of polytope theory, was conjectured by Motzkin in 1957 \cite{motzkin1957comonotone} and proved by McMullen in 1970 \cite{mcmullen1970maximum}. This theorem states that for any convex $d$-polytope $P$ with $n$ vertices, the maximum number of $i$-faces, denoted $f_i$, is achieved by the cyclic polytope $C_d(n)$. Formally, $f_i(P) \leq f_i(C_d(n))$ for all $i$, and any $d$-polytope $P$ with $n$ vertices (see e.g., \cite{ziegler2012lectures}). We are now ready to prove the following result.

\begin{theorem}\label{Thm:UpBound}
   Let $P$ be a $d$-polytope with $n$ vertices in $\mathbb{R}^d$. Then 
    \[
    \nu(P) \ \leq \ \nu(C_d(n)) \ = \ 
    \begin{cases} 
         2(n-2) & \text{if } d = 3, \\
         \lfloor n/2 \rfloor \lceil n/2 \rceil & \text{if } d > 3.
    \end{cases}
    \]
\end{theorem}

\begin{figure}[ht]
    \centering
    \input{Figures/cyclic.tikz}
\caption{Visualization of the $1$-skeletons of $C_3(7)$ and $C_4(7)$, alongside two hyperplanes (shown as purple lines) positioned to maximize the number of edge intersections in each case.}
\label{fig:cyclic}
\end{figure}

\begin{proof}[ Proof of Theorem \ref{Them:max vertices}]
    By the Upper Bound Theorem, the cyclic polytope $C_d(n)$ maximizes the number of faces of each dimension, in particular, the number of edges among all $d$-polytopes with $n$ vertices. According to Lemma \ref{lemma 1}, the value of $\nu(P)$ is achieved by hyperplane slices that intersect the interior of the largest possible number of edges of $P$. Therefore, to bound $\nu(P)$, it suffices to bound $\nu(C_d(n))$.

From Lemma~\ref{lemma 1}, we know that for $d = 3$, the maximum number of vertices that can appear in a slice of the cyclic polytope is $2(n - 2)$.

For dimensions $d > 3$, this maximum is given by $\lfloor n/2 \rfloor \lceil n/2 \rceil$, and it is attained when the slicing hyperplane $H$ is a \emph{bisecting hyperplane} (see Figure~\ref{fig:cyclic}); that is, $H$ divides the vertices so that $\lfloor n / 2 \rfloor$ lie in one open half-space and $\lceil n / 2 \rceil$ in the other.\\
\end{proof}

\section{On the Vertex Slice Sequences for 3-polytopes}\label{Sec:GapsR3}

An interesting experimental phenomenon we have observed is that nearly all $3$-polytopes appear to have a complete VSS, which means that all numbers from $1$ to $\nu(P)$ appear in the sequence. In fact, four out of the five Platonic solids satisfy this property, the icosahedron being the only exception, which lacks a slice with exactly four vertices. Next, we will prove the existence of an infinite family of $3$-polytopes with gaps.

\begin{proposition}\label{Proposition:existence}
Let $k,d\in \mathbb{N}$ with $k\geq d$, and $P$ be a $d$-polytope. If the $1$-skeleton of $P$ is $k$-connected, and $P$ has no faces with exactly $r$ vertices, with $r < k$. Then, $r$ is a gap for $P$.
\end{proposition}

\begin{proof}
Let $H$ be a hyperplane intersecting $P$. If $H$ is tangent to $P$, then $H \cap P$ is a face of $P$. By hypothesis, this face cannot have exactly $r$ vertices, so $r$ is not realized in this case. On the other hand if $H$ is not tangent to $P$, then $H$ strictly separates a  pair of vertices of $P$ each lying in one of the half-spaces defined by $H$. Since the $1$-skeleton of $P$ is $k$-connected, for any pair of vertices in different half-spaces, there exist $k$ disjoint paths connecting them. Each of these paths must intersect $H$, and each intersection contributes a vertex to the slice $H \cap P$. 
Because there are at least $k$ such paths, the slice $H \cap P$ must have at least $k$ vertices. Since $r < k$, it follows that $H \cap P$ cannot have exactly $r$ vertices. 

\end{proof}





In dimension three, the icosahedron satisfies the hypotheses of Proposition \ref{Proposition:existence} for $k = 5$ and $r = 4$ (see Figure \ref{ej_triangulation}).

\begin{figure}[ht]
\centering
\input{Figures/IcosahedronPaths.tikz}
\caption{For every combinatorially different pair of vertices (points in red) in the icosahedron, there exist five disjoint paths (shown in color) connecting them.}
\label{icosahedron paths}
\end{figure}

A construction of an infinite family of graphs that are 5-connected, simple, planar and without quadrangular faces can be found in \cite{barnette1974generating}. Therefore, by Steinitz's theorem, all the embeddings of these graphs have 4 as a gap. In the next part we show a different family.

\subsection{VSS of Stacked Polytopes}

Now we define an operation $\sigma$ on the space of convex $d$-polytopes. Given a polytope $P$, $\sigma(P)$ is constructed by stacking a pyramid onto every facet of $P$. Specifically, for each facet $F$ of $P$, we add a new vertex that lies strictly beyond $F$ but beneath all other facets, and then take the convex hull.
We call $\sigma(P)$ the \emph{all-facets stacked polytope induced by} $P$.

Note, this operation is slightly more general than the classical \emph{stacked polytopes}, which are built by sequentially stacking simplices onto the facets of a starting simplex (see \cite{gonska2011inscribable,ziegler2012lectures}). Simplicial stacked polytopes are fundamental in polytope theory because they minimize the number of faces among simplicial $d$-polytopes with a given number of vertices, achieving equality in the Lower Bound Theorem (\cite{barnette1971minimum,barnette1973proof,grunbaum2003convex}).



\begin{figure}[ht]
\begin{center}
\input{Figures/T_t.tikz}
\caption{A regular tetrahedron $T$ and the stacked tetrahedron $\sigma(T)$.}
\label{ej_triangulation}
\end{center}
\end{figure}


\begin{proposition}\label{Proposition:BeforeGapsR3}
Let $T$, $C$, $O$, and $I$ denote the regular tetrahedron, cube, octahedron, and icosahedron, respectively. Then:
\begin{itemize}
    \item $\sigma(T)$ does not have slices with four vertices,
    \item $\sigma(C)$ and $\sigma(O)$ do not admit slices with five vertices,
    \item $\sigma(I)$ does not have slices with four, six, seven, eight or nine vertices.
\end{itemize}
\end{proposition}

\begin{proof}
Let $H$ be a hyperplane intersecting $\sigma(P)$, where $P$ is one of the polytopes $T$, $C$, $O$, or $I$. There are two cases:

\begin{enumerate}
    \item \textbf{$H$ intersects the relative interior of $P$:} \\
    Since each face of $P$ is stacked, $H$ intersects at least two facets of $\sigma(P)$ for every facet of $P$ that it intersects.

    \item \textbf{$H$ does not intersect the relative interior of $P$:} \\
    In this scenario, $H$ either completely contains $F$, intersects at most one of its edges, or does not intersect $F$ at all (otherwise, it would intersect the relative interior of $P$). Consequently, the slice determined by $H$ on $\sigma(P)$ contains one, two, or exactly the number of vertices of $F$.
\end{enumerate}

We analyze each polytope separately:

\vspace{0.2cm}

\textbf{The Tetrahedron ($T$):} \\
If $H$ intersects the relative interior of $T$, it intersects either one edge and two faces of $T$, or three faces. As a result, the slice in $\sigma(T)$ has either five or six vertices. If $H$ does not intersect the relative interior of $T$, the slice has one, two, or three vertices.  
Thus, $\sigma(T)$ does not have slices with four vertices.

\vspace{0.2cm}

\textbf{The Cube ($C$):} \\
If $H$ intersects the relative interior of $C$, then $H$ intersects at least two facets. Furthermore, intersections that result in triangular slices intersect three facets of $C$. Consequently, the resulting slices in $\sigma(C)$ have at least six vertices.  
If $H$ does not intersect the relative interior of $C$, the slices may have one, two, three, or four vertices.  
Thus, $\sigma(C)$ does not have slices with five vertices.

\vspace{0.2cm}

\textbf{The Octahedron ($O$):} \\
Planes $H$ that intersect the relative interior of $O$ do so in one of two ways: they either pass through four vertices without intersecting any face or they pass through at least three faces. Consequently, the slices in $\sigma(O)$ have four vertices or more than six.  
If $H$ does not intersect the relative interior of $O$, the slices may have one, two, three, or four vertices.   
Thus, $\sigma(O)$ does not have slices with five vertices.

\vspace{0.2cm}

\textbf{The Icosahedron ($I$):} \\
Planes $H$ intersecting the relative interior of $I$ do so in one of three ways:  
1. They pass through five vertices of $I$ without intersecting any face.  
2. They pass through exactly two edges and four faces.  
3. They pass through at least five faces.  
In these cases, the slices in $\sigma(I)$ have five, ten, or more than ten vertices, respectively.  
If $H$ does not intersect the relative interior of $I$, the slices may have one, two, three or five vertices.  
Thus, $\sigma(I)$ does not have slices with four, six, or seven vertices.
\end{proof}

\begin{lemma}\label{lemma:sigma_sigma}
    If $d+1$ is a \emph{gap} in the \emph{vertex slice sequence} of a $d$-polytope $P$, then it is also a gap in the \emph{vertex slice sequence} of $\sigma(P)$.
\end{lemma}

\begin{proof}
    Let $H\in \mathcal{H}_{\sigma(P)}$. First, suppose that $H$ is tangent to $P$ at a facet $F$ with $n$ vertices. In this case, by convexity $H$ intersects $\sigma(P)$ in exactly the same $n \neq d+1$ vertices. 
    
    If $H$ intersects the relative interior of $P$, then it intersects the relative interior of at least two facets of $P$. Since $\sigma(P)$ adds a new vertex and new edges on each facet, the hyperplane $H$ will intersect at least two of these new elements (one for each intersected facet), which implies that the slice $H \cap \sigma(P)$ contains at least two new vertices, in addition to those already in $H \cap P$. Consequently, the total number of vertices in the slice increases by at least two, and since the $1$-skeleton of $P$ is $d$-connected, the number of vertices in $H \cap \sigma(P)$ is at least $d+2$.

    Finally, suppose that $H$ does not intersect $P$ but does intersect $\sigma(P)$. Then $H$ must intersect the pyramid constructed over some facet $F$ of $P$. If $H$ contains the new vertex $v$ (the vertex added over $F$), then $H$ is tangent to $\operatorname{conv}(F \cup \{v\})$, and the slice has at most $d$ vertices. On the other hand, if $H$ separates $v$ from $F$, then the vertices of the section come from the edges connecting $v$ to the vertices of $F$, and the slice contains exactly $\vert F\vert$ vertices. In both cases, by hypothesis, the number of vertices in the slice cannot be $d+1$.
\end{proof}

\begin{proof}[Proof of Theorem~\ref{Thm:GAPSR3}]   
By Proposition~\ref{Proposition:BeforeGapsR3}, four is a gap in the VSS of $\sigma(T)$. Applying Lemma~\ref{lemma:sigma_sigma} iteratively, it follows that four remains a gap in the VSS of $\sigma^n(T)$ for every $n \in \mathbb{N}$. Moreover, since there are always vertices of degree three, the 1-skeleton of $\sigma^n(T)$ is 3-connected but not 4-connected for all $n$.
\end{proof}

\section{A Poset to Identify All Slices}\label{sec:Posets}
Regarding the task of listing all slices, particularly those with the maximum number of vertices, we now present an alternative technique that offers a combinatorial perspective on the distinct slices, differing from the approach proposed in \cite{brandenburg2025best}. This method uses partially ordered sets (posets) on the edges and vertices of the polytope to analyze not only the slices with the maximum number of vertices but also all possible slices. This poset structure was first suggested by O'Neil \cite{o1971hyperplane} for the hypercube, and we now present a generalization of his idea. Although this technique may be computationally expensive and still relies on the methods of \cite{brandenburg2025best} to enumerate all directions $u$, it provides a deeper combinatorial understanding of the slices. It is also interesting that our poset is almost dual to the poset of oriented 1-skeleta of polytopes presented in \cite{Hersh24,gaetzhersch}.

In what follows, if $H=\{u\cdot X=t\}$ is a hyperplane, denote $H^-=\{u\cdot X<t\}$ and $H^+=\{u\cdot X>t\}$ the two open half-spaces induced by $H$. We begin with an important remark regarding polytopes and hyperplanes.

\begin{proposition}\label{Proposition:Neighbors}
Let $P$ be a $d$-polytope and $H\in\mathcal{H}_P$. If $v \in V(P) \cap H$ and $H^\pm \cap P \neq \emptyset$, where $H^\pm$ denotes either $H^-$ or $H^+$, then there exists a vertex $w \in V(P) \cap H^\pm$ adjacent to $v$, and moreover, the subgraph of the 1-skeleton of $P$ induced by the vertices in $H^\pm$ is connected.
\end{proposition}

\begin{proof}
Suppose that no neighbor of $v$ lies in $H^\pm$. Since $v$ is a vertex, $P$ is contained in the cone 
\[ C_v = v + \text{cone}\{v_i - v \mid v_i \text{ is a neighbor of } v\}. \] 
By assumption, all generators $v_i - v$ point into $\mathbb{R}^d \setminus H^\pm$, so $C_v \subset \mathbb{R}^d \setminus H^\pm$. However, any $p \in P \cap H^\pm$ must be reachable through directions entering $H^\pm$, yielding a contradiction. Thus, some neighbor $w$ must lie in $H^\pm$.

For the connectivity part, if there is only one vertex of $P$ in $H^\pm$, its 1-skeleton is trivially connected. Suppose there is more than one vertex of $P$ in $H^\pm$. Consider a point $x$ in the relative interior of $P \cap H$. We construct the polytope
\[
P^\pm_x = \operatorname{conv}\big((V(P) \cap H^\pm) \cup \{x\}\big).
\]
Let $r \geq 2$ be the dimension of $P^\pm_x$. Its 1-skeleton is $r$-connected and remains connected after removing $x$. This proves that the subgraph induced by $V(P) \cap H^\pm$ is connected.

\end{proof}

Let $P$ be a $d$-polytope and $u\in \mathbb{R}^d$ be a unit vector. Let $v_1, v_2, \dots, v_n$ be the vertices of $P$, ordered by direction $u$, i.e., $i<j$ if $u\cdot v_i\leq u\cdot v_j$. If there is an edge between $v_i$ and $v_j$, we denote it $\{v_i, v_j\}$.\\

We now define the \emph{slicing poset} $(P_u, \leq)$, whose elements are pairs representing either vertices or edges of $P$, defined by:
$$P_u=\{(v_i, v_j): i=j, \text{ or $\{v_i, v_j\}$ is an edge such that $u\cdot (v_i-v_j)\neq 0$, with $i<j$ } \}.$$
Here $(v_i,v_i)$ represents the vertex $v_i$, while $(v_i,v_j)$ with $i<j$ represents the edge $\{v_i,v_j\}$.\\

The order on $P_u$ is defined as follows: two elements of $P_u$ are equal if they represent geometrically the same element, otherwise they are different elements. Let $(v_{i_1}, v_{j_1}), (v_{i_2}, v_{j_2})\in P_u$, we say that $(v_{i_1}, v_{j_1}) < (v_{i_2}, v_{j_2})$ if either:
\begin{enumerate}
    \item $j_1 = i_2$; or
    \item $j_1 < i_2$ and there exists a path $v_{j_1} = v_{r_1}, \dots, v_{r_m} = v_{i_2}$ such that $u\cdot {v_{r_s}} < u\cdot v_{r_{s+1}}$ for all $s = 1, \dots, m-1$.
\end{enumerate}

It is straightforward to verify that $(P_u,\leq)$ is indeed a poset, and all maximal and minimal elements are vertices. We show an example of an slicing poset of the octahedron in Figure \ref{fig:partial_order_cube}. 
\begin{figure}[ht]
    \centering
    \input{Figures/partial_order_cube.tikz}
    \caption{The octahedron $O$ with vertices $\{a, b, c, d, e, f\}$ (left), the poset induced by direction $u$ (center), and the poset induced by direction $w$ (right).
}
    \label{fig:partial_order_cube}
\end{figure}

    
    


    \begin{proof}[Proof of Theorem \ref{Thm:antichain}]
    Let $P$ be a $d$-polytope and $u\in \mathbb{R}^d$. We show that for each $u$ and any hyperplane $H\in\mathcal{H}_P$ with normal vector $u$, the set of elements of $P_u$ intersected by $H$ forms a maximal antichain in the poset $(P_u, \leq )$ defined above.

    Let $H= \{X: u\cdot X=t\}$ and let $S\subset P_u$ be the set of elements that $H$ intersects. Let $v_1,\dots,v_n$ be the vertices of $P$ ordered by $u$ and denote $t_i:=u\cdot v_i$. We begin by proving that $S$ is an antichain. If $\vert S\vert=1$, we are done. Suppose $\vert S\vert>1$.
    
    Let $(v_{i_1}, v_{j_1}), (v_{i_2}, v_{j_2})\in S$ be two distinct elements. Then $t_{i_1}\leq t \leq t_{j_1}$ and $t_{i_2}\leq t \leq t_{j_2}$ in order to be intersected by $H$ where $t_{i_1}=u\cdot v_{i_1}$, $t_{i_2}=u\cdot v_{i_2}$, $t_{j_1}=u\cdot v_{j_1}$, y $t_{j_2}=u\cdot v_{j_2}$. Observe that it is not possible to have $t_{j_1}<t_{i_2}$ and $t_{j_2}<t_{i_1}$. 
    
    Proceeding by contradiction, suppose without loss of generality that $(v_{i_1}, v_{j_1})<(v_{i_2}, v_{j_2})$. Then $j_1\leq i_2$, which implies $t_{j_1}\leq t_{i_2}$. The only way this can happen is if $j_1=i_2$, by the previous observation.
    
     Since $t=t_{j_1}= t_{i_2}=t$, both $(v_{i_1}, v_{j_1})$ and $(v_{i_2}, v_{j_2} )$ must be vertices, otherwise $H$ would not intersect an edge in its interior. Therefore, $$(v_{i_1}, v_{i_1})=(v_{i_1}, v_{j_1})=(v_{i_2}, v_{j_2})=(v_{i_2}, v_{i_2}),$$ which contradicts the assumption that they are distinct elements.
    
    Now we prove that $S$ is a maximal antichain. Suppose, for contradiction, that $S$ is not maximal. Then there exists $(v_i, v_j)\in P_u\setminus S$ that is incomparable with every element of $S$ (note that $i$ may equal $j)$. Observe that either $t<t_i$ or $t_j<t$  for $t_i=u\cdot v_i$ and $t_j=u\cdot v_j$. Let $H_{i}=\{X: uX=t_i\}$ and $H_{j}=\{X: uX=t_j\}$ be hyperplanes. 

    If $t<t_i$, suppose without loss of generality that $(v_i, v_j)$ is the minimal element in $(P_u,\leq)$ with these conditions. Since $H\subset H_i^-$ and $H\cap P\neq \emptyset$, then $H_{i}^-\cap P\neq \emptyset$. By Proposition \ref{Proposition:Neighbors}, there exists a vertex $v_r\in V(P)\cap H_i^-$ adjacent to $v_i$, which implies $r<i$ because $t_r<t_i$. Therefore, $(v_r,v_i)<(v_i, v_i)\leq (v_i, v_j)$. 

    Since $(v_r,v_i)$ is comparable to some element of $S$, say $(v_{i_1}, v_{j_1})\in S$, we must have $(v_{i_1}, v_{j_1})\leq (v_r,v_i)$, otherwise $t_i<t_{i_1}\leq t$, which leads to a contradiction. Finally, $$(v_{i_1}, v_{j_1})\leq (v_r,v_i)<(v_i, v_i)\leq (v_i, v_j),$$ a contradiction.

    A similar argument applies if $t<t_j$, suppose without loss of generality that $(v_i, v_j)$ is the maximal element in $(P_u,\leq)$ with these conditions. Since $H\subset H_j^+$ and $H\cap P\neq \emptyset$, then $H_j^+\cap P\neq \emptyset$. By Proposition \ref{Proposition:Neighbors}, there exists a vertex $v_r\in V(P)\cap H_j^+$ adjacent to $v_j$, which implies $j<r$ because $t_j<t_r$. Therefore, $(v_i,v_j)\leq(v_j, v_j)< (v_j, v_r)$. 

    Since $(v_j,v_r)$ is comparable to some element of $S$, say $(v_{i_1}, v_{j_1})\in S$, we must have $(v_j,v_r)\leq (v_{i_1}, v_{j_1})$, otherwise $t_j<t_{j_1}\leq t$ leads to a contradiction. Finally, $$(v_i, v_j)\leq (v_j,v_j)\leq (v_j, v_{i_1})\leq (v_{i_1}, v_{j_1}),$$ which is a contradiction.
\end{proof}

\begin{remark}\label{Remark:Hypercube}
     In \cite{o1971hyperplane}, the author order the edges of $Q_d$ as $$E(Q_d)^*=\{(v,w): v,w\in V(Q_d), \text{ and $w$ has more ones than $v$},\}$$ and defines a poset $(E(Q_d)^*, <^* )$ as follows: Let $(v_1,w_1), (v_2, w_2)\in  E(Q_d)^*$, then $(v_1,w_1)<^*(v_2, w_2)$ if $v_2$ has a one in all entries where $w_1$ has a one. We call this poset \emph{O'Neil's poset} in dimension $d$.

     We extend O'Neils's posets to the vertices as follows: Let $$Q_d^*=\{(v,v): v\in V(Q_d) \}\cup E(Q_d)^*, $$ where $(v, v)$ denotes the vertex $v$. Let $(v_1,w_1), (v_2, w_2)\in Q_d^* $ be two different elements, then $(v_1,w_1)<^*(v_2, w_2)$ if $v_2$ has one in all the entries where $w_1$ has a one. Let us call to this poset the \emph{extended O'Neil's poset} in dimension $d$, and we denote it by $(Q_d^*,\leq ^*)$.
\end{remark}

\begin{lemma}\label{Lemma:PlusOne}
    Let $u \in \mathbb{S}^{d-1}$ with positive entries, and let $v_1, v_2, \dots, v_n$ be the vertices of $Q_d$, ordered by the direction $u$. If $v_j$ has one in all entries where $v_i$ has a one, with $i\neq j$, then $j>i$ and $u\cdot v_j>u\cdot v_i$.  
\end{lemma}
\begin{proof}
    Let $v_i=(v_1^{i},\dots, v_d^{i})$ and $v_j=(v_1^{j},\dots, v_d^{j})$ and $u=(u_1,\dots,u_d)$. By hypothesis, $0<u_r$ and $0\leq v_r^{i}\leq v_r^{j}$ for $r=1,\dots, d$, and there is at least one $r$ where the equality does not hold (because $i\neq j$). Then  \begin{equation*}
        u\cdot v_i=\sum_{r=1}^d u_r v_r^{i}<\sum_{r=1}^d u_r v_r^{j}=u\cdot v_j.
    \end{equation*}

    Finally, $i<j$ because the vertices are ordered by $u$.
\end{proof}

\begin{proposition}\label{Proposition:equivalenceQd}
    For every $u \in \mathbb{S}^{d-1}$ with positive entries and every $d\geq2$. The extended O'Neil's poset in dimension $d$ is equivalent to the slicing poset of the $d-$hypercube $Q_d$ in direction $u$.
\end{proposition}
\begin{proof}
    Let $u \in \mathbb{S}^{d-1}$ with positive entries, let $v_1\dots v_n$ be the vertices of $Q_d$ ordered by direction $u$, and denote $t_i=u\cdot v_i$. 
    
    We begin by proving that $Q_d^*= (Q_{d})_u$. Let $1\leq i\leq j\leq n$. If $i=j$, then $(v_i,v_j)\in Q_d^*$ if and if $(v_i,v_j)\in (Q_{d})_u$. Let us prove the case $i<j$, which implies $t_i\leq t_j$. 
    
    Let $(v_i,v_j)\in (Q_d)_u$. Then $(v_i,v_j)$ represents the edge $\{v_i, v_j\}$, and since $u\cdot (v_i-v_j)\neq 0$, $t_i<t_j$. We claim that $v_j$ has one in all entries where $v_i$ has one, plus another one in some entry where $v_i$ has a zero. This would imply that $(v_i,v_j)\in Q_d^*$

    Indeed, since $\{v_i, v_j\}$ is an edge, then $v_i$ and $v_j$ must differ only in one entry, with $v_j$ having the extra one. Otherwise, by Lemma \ref{Lemma:PlusOne}, $t_i>t_j$ which implies $i>j$, leading to a contradiction. 

    Now, if $(v_i,v_j)\in Q_d^*$, then $v_j$ has one in all entries where $v_i$ has one, plus another one in some entry where $v_i$ has a zero, which, according to Lemma \ref{Lemma:PlusOne}, is consistent with the fact that $i<j$. Since $\{v_i, v_j\}$ is an edge of $Q_d$, then $(v_i,v_j)\in (Q_d)_u$.
    
   To conclude the proof, we must show that for any two distinct elements $(v_{i_1},v_{j_1}), (v_{i_2}, v_{j_2})\in (Q_d)_u=Q_d^*$, we have $(v_{i_1},v_{j_1})<^*(v_{i_2}, v_{j_2})$ if and only if $(v_{i_1},v_{j_1})<(v_{i_2}, v_{j_2})$. 
   
   Suppose $(v_{i_1},v_{j_1})<^*(v_{i_2}, v_{j_2})$. By definition, $v_{i_2}$ has a one in all components where $v_{j_1}$ has a one. Then, by Lemma \ref{Lemma:PlusOne}, $i_2\geq j_1$. If $i_2=j_1$, we have $(v_{i_1},v_{j_1})<(v_{i_2}, v_{j_2})$ by definition.
    
    If $i_2>j_1$, let $v_{j_1}=v_{r_1}\dots v_{r_m}=v_{i_2}$ be a path where each $v_{r_{s+1}}$ has a one in all components where $v_{r_s}$ has a one, plus an extra one in the entry where $v_{i_2}$ has a one and $v_{j_1}$ a zero, for all $s=1,\dots, m-1$. Then, by Lemma \ref{Lemma:PlusOne}, $t_{r_s}< \dots <t_{r_{s+1}}$, which implies $(v_{i_1},v_{j_1})<(v_{i_2}, v_{j_2})$.

    Conversely, suppose $(v_{i_1},v_{j_1})<(v_{i_2}, v_{j_2})$. We consider two cases. If $j_1=i_2$, then $(v_{i_1},v_{j_1})<^*(v_{i_2}, v_{j_2})$ by definition. If $j_1<i_2$, there is a path $v_{j_1}=v_{r_1}\dots v_{r_m}=v_{i_2}$, with $t_{r_s}< \dots <t_{r_{s+1}}$ for all $s=1,\dots, m-1$. By Lemma \ref{Lemma:PlusOne}, and the fact that $v_{r_s}$ and $v_{r_{s+1}}$ must differ in only in one entry (because $\{v_{r_s}, v_{r_{s+1}}\}$ is an edge), we have that $v_{r_{s+1}}$ has a one in all entries where $v_{r_s}$ has a one. Then, by induction, $v_{i_2}$ has a one in all components where $v_{j_1}$ has a one. Therefore, $(v_{i_1},v_{j_1})<^*(v_{i_2}, v_{j_2})$.
    
\end{proof}

\begin{corollary}\label{Corollary:NonZero}
    For every $u \in \mathbb{S}^{d-1}$ with no zero entries and every $d\geq2$. The extended O'Neil's poset in dimension $d$ is equivalent to the slicing poset of the $d-$hypercube $Q_d$ in direction $u$.
\end{corollary}
\begin{proof}
    Let $v_1,\dots, v_n$ be the vertices of $Q_d$ ordered by $u$. If all the entries of $u=(u_1,\dots,u_d)$, the proof is complete by Proposition \ref{Proposition:equivalenceQd}. If $u$ has a negative entry, say in entry $i$, define the transformation $T_i:\mathbb{R}^d\rightarrow \mathbb{R}^d$ as $T_i(x_1,\dots, x_i, \dots, x_n)=(x_1,\dots, 1-x_i, \dots, x_n)$. 

    Note that $T_i$ is a reflection about the hyperplane $H_i=\{X=(x_1,\dots, x_d): x_i=\frac{1}{2}\}$. Thus, $T_i(Q_d)=Q_d$ and the direction $u$ under the transformation $T_i$ is the direction $u^i=(u_1,\dots, -u_i, \dots, u_d).$ It can be verified that $((Q_d)_{u^i},\leq)\cong((Q_{d})_u, \leq)$, by relabeling the vertices of $((Q_{d})_u, \leq)$ using $T_i$.

    We may repeat this process a finite number of times until we obtain a direction $u'$ with all positive entries. Then $$((Q_{d})_u, \leq)\cong (Q_d)_{u'},\leq)\cong (Q_d^*, <^*),$$ where the last equivalence follows from Proposition \ref{Proposition:equivalenceQd}.
\end{proof}

\section{Slices of the Hypercube}\label{Section:Hypercube}

Since the $d-$hypercube $Q_d$ has a rich combinatorial structure, we present some values of the function $cv_{Q_d}$, gaps and the VSS for the low dimensional hypercubes. We start by giving some general results that will be helpful to prove the last two main theorems.

\begin{proposition} \label{Pro:ZeroEntries}
    Let $u\in\mathbb{S}^{d-1}$  with $i$ entries equal to zero. Denote $Q^1_{d-i}\subset Q_d$ the hypercube whose vertices have zero in all the entries where $u$ does. Let $H\in \mathcal{H}_{Q_d}^d$, with normal $u$. Then $cv_{Q_d}(H)=2^{i}\cdot|V(H \cap Q^1_{d-i})| $, $|H^- \cap V(Q_d)|=2^{i}\cdot|H^- \cap V(Q^1_{d-i})|$, and $|H^+ \cap V(Q_d)|=2^{i}\cdot|H^+ \cap V(Q^1_{d-i})|$.
\end{proposition}
\begin{proof}
    Let $I$ be the set of indices where $u$ has zero entries, so $\vert I \vert=i$. Fix a particular combination of zeros and ones in the coordinates indexed by $I$. Each of these combinations defines a $(d-i)$-hypercube contained in $Q_d$. There are $2^{i}$ of such combinations, giving rise to $2^{i}$ pairwise disjoint hypercubes $Q_{d-i}^1,\dots, Q_{d-i}^{2^{i}}$ such that:
     $$V(Q_d)=\bigcup_{j=1}^{2^{i}} V(Q_{d-i}^j).$$

    Now, take the following partition of $V(Q_d)$: for each vertex $v\in V(Q_{d-i}^1)$, let
    $$V_{v}=\{w\in V(Q_d): \text{the entry $j$ of $w$ and $v$ have the same value, for each $j\notin I$}\}.$$
    Then $$V(Q_d)=\bigcup _{v\in V(Q_{d-i}^1)}\in V_{v}$$
    
    The following are straightforward implications. Some of them need the fact that $H=\{X: X\cdot u=t\}$, for some $t\in \mathbb{R}$.

    \begin{enumerate}
        \item For each $v\in V(Q_{d-i}^1)$, $V_{v}$ has exactly one vertex of each $Q_{d-i}^j$, for $j=1,\dots, 2^{i}$.
        \item If $v\in V(Q_{d-i}^1)\cap H$, then $w\in H$, for all $w\in V_{u}$. 
        \item If $v\in V(Q_{d-i}^1)\cap H^*$, then $w\in H^*$, for $*=\pm$ and all $w\in V_{u}$.
        \item If $\{v_1, v_2\}$ is an edge of $Q^{1}_{d-i}$. Then, for all $j=1,\dots, 2^{i}$, $\{w^j_1, w^j_2\}$ is an edge, with $w^j_1= V(Q^{j}_{d-i})\cap V_{w^j_1}$ and $w_2= V(Q^{j}_{d-i})\cap V_{v_2}$. Furthermore, by the clause before, if $H$ intersects $\{v_1, v_2\}$, then $H$ intersects $\{w^j_1, w^j_2\}$.
    \end{enumerate}
    These facts imply that every intersection of H with $Q_d$ corresponds to exactly $2^{i}$ equivalent intersections with the hypercubes $Q^{j}_{d-i}$. This proves the proposition.
    
\end{proof}

\begin{proposition}\label{prop:even_cut}
Let $d$ be an even integer. Then, for all $H\in\mathcal{H}_{Q_d}$ that contain no vertices of $Q_d$, $cv_{Q_d}(H)$ is even.
\end{proposition}

\begin{proof}

We recall some basic definitions in graph theory (see \cite{bondy2008graph} for more details). A \emph{circuit} is a sequence of adjacent vertices that starts and ends at the same vertex. Circuits never repeat edges, but they may repeat vertices in the sequence. 

Let $G=(V,E)$ be a connected graph. We say that $G$ contains an \emph{Eulerian circuit} if there exists a circuit in $G$ that contains all the edges of $G$. Additionally, let $V=V_1\cup V_2$ a partition of the vertices of $G$. The cutset $(V_1,V_2)$ of $G$ is the set of edges of $G$ that have one end points in $V_1$ and the other in $V_2$.

Two classical results. Let $G$ be a connected simple graph, then

\begin{enumerate}
    \item $G$ has a Eulerian circuit if and only all its vertices have even degree.
    \item Any circuit and any cutset of $G$ have an even number of common edges.
\end{enumerate}

 We proceed to the proof. Let $H\in\mathcal{H}_{Q_d}$ that does not contain any vertices of $Q_d$. This implies that $H$ is not tangent to $Q_d$, then each of the half-spaces, $H^+$ and $H^-$, contain vertices of $Q_d$. Let $V_1=V(Q_d)\cap H^-$ and $V_2=V(Q_d)\cap H^+$. Note that $H$ only intersect edges that are in the cutset $(V_1,V_2)$.

 Since every vertex of the $1$-skeleton of $Q_d$ have even degree $d$, the $1$-skeleton has an Eulerian circuit $C$.

 By the above classic result, $C$ and $(V_1,V_2)$ have an even number of common edges. Therefore, $\vert (V_1,V_2)\vert$ is even, because $C$ contains all the edges of the $1$-skeleton of $Q_d$.
\end{proof}

\begin{proposition}\label{Proposition:2k}
 Let $H'\in\mathcal{H}_{Q_d}$ be a hyperplane slicing $Q_d$ into a polytope with $k$ vertices. Then, there exist a hyperplane $H\in\mathcal{H}_{Q_{d+1}}$ that slices $Q_{d+1}$ into a polytope with $2k$ vertices.   
\end{proposition}
    
\begin{proof}
A hyperplane \(H\) divides \(\mathbb{R}^d\) into three convex regions: two open half-spaces, \(H^+\) and \(H^-\), and \(H\) itself. Given a set of points \(\{p_1, \ldots, p_{d}\}\) that generate \(H\), consider the following function:

\[
\phi(p_1, \ldots, p_{d}, v) = \text{sign}\left(\det\begin{bmatrix} (p_1, 1) \\ \vdots \\ (p_{d}, 1) \\ (v, 1) \end{bmatrix}\right).
\]

This function determines the region where a point \(v \in \mathbb{R}^d\) is contained based on its sign: if it is positive, then \(v \in H^+\); if it is negative, then \(v \in H^-\); and \(v \in H\) otherwise.

Now, let \(H' \subset \mathbb{R}^d\) be a hyperplane slicing \(Q_d\), and let \(\{p_1, \ldots, p_{d}\} \subset \mathbb{R}^{d}\) be a set of points that generate \(H'\). Consider the hyperplane \(H \subset \mathbb{R}^{d+1}\) generated by the points \(\{(p_1,0), \ldots, (p_{d},0)\} \cup \{(p_1,t)\}\) for some \(t \neq 0 \in \mathbb{R}\). For any \(v \in V(Q_{d})\), the vertices \((v,1)\) and \((v,0)\) of \(V(Q_{d+1})\) satisfy:

\begin{align*}
   \phi((p_1,0), \ldots, (p_{d},0), (p_1,t), (v,0)) &= \text{sign}\left(\det\begin{bmatrix} (p_1,0, 1) \\ \vdots \\ (p_{d},0, 1) \\ (p_1,t,1)\\(v,0, 1) \end{bmatrix}\right) \\
   &= \text{sign}\left( t \det\begin{bmatrix} (p_1, 1) \\ \vdots \\ (p_{d}, 1)\\(v, 1) \end{bmatrix}\right) \\
   &= \text{sign}\left(\det\begin{bmatrix} (p_1,0, 1) \\ \vdots \\ (p_{d},0, 1) \\ (p_1,t,1)\\(v,1, 1) \end{bmatrix}\right) \\
   &= \phi((p_1,0), \ldots, (p_{d},0), (p_1,t), (v,1)).
\end{align*}

This implies that \((v,0)\) and \((v,1)\) are contained in the same region defined by \(H\). Therefore, for every edge crossing \(H'\) in \(Q_d\), there are two edges crossing \(H\) in \(Q_{d+1}\). Similarly, for every vertex of \(Q_d\) contained in \(H'\), there are two vertices of \(Q_{d+1}\) contained in \(H\).
\end{proof}

\subsection{Gaps of the hypercubes}

In this subsection, we prove Theorem \ref{Thm:FirstGAPS} and Theorem \ref{Thm:gap_penultimo}.

 Let $v=(x_1,\dots, x_d)$ be a vertex of the $d$-dimensional hypercube $Q_d$. Then $a(v)=(y_1,\dots, y_d)$, with $y_i=x_i+1\mod 2$, is called the \emph{antipodal} vertex of $v$. Observe that there are $2^{d-1}$ pairs of antipodal vertices and these pairs induce a partition on $V(Q_d)$.

Let $P$ be a $d$-polytope and $v \in V(P)$. The set $N(v) \subset V(P)$ of \emph{neighbors} of $v$ consists of all vertices adjacent to $v$ in the 1-skeleton of $P$. Note that if $P=Q_d$, then for every $v_1, v_2\in V(Q_d)$, the size of the intersection $\vert N(v_1)\cap N(v_2)\vert$ is either zero or two. 

\begin{lemma}\label{Lemma:NoAntipodal}
    Let $Q_d$ be the $d$-dimensional hypercube, and suppose $\vert H^-\cap V(Q_d)\vert \leq \vert H^+\cap V(Q_d)\vert$, for $H\in \mathcal{H}_{Q_d}$. Then, for every $v\in H^-\cap V(Q_d)$, $a(v)\in H^+\cap V(Q_d)$.
\end{lemma}
\begin{proof}
    Let $m=(\frac{1}{2},\dots,\frac{1}{2})$, which is the midpoint of the segments defined by any pair of antipodal points. Observe the following cases:
    
    \textbf{Case 1:} If $m\in H$, then clearly $v\in H^-\cap V(Q_d)$ if and only if $a(v)\in H^+\cap V(Q_d)$.
    
    \textbf{Case 2:} If $m\in H^+$, then for all $v\in H^-\cap V(Q_d)$, $a(v)\in H^+\cap V(Q_d)$, in order for $m$ to be contained in the segment $va(v)$.
    
    \textbf{Case 3:} If $m\in H^-$, for every $v\in V(Q_d)\cap (H\cup H^+)$, $a(v)$ must be in $H^-$ in order for $m$ to be contained in the segment $va(v)$. This implies $\vert H^-\cap V(Q_d)\vert \geq \vert (H\cup H^+)\cap V(Q_d)\vert$. This is only possible if $\vert H^-\cap V(Q_d)\vert=\vert H^+\cap V(Q_d)\vert$ and $\vert H\cap V(Q_d)\vert=0$. Therefore, $v\in H^-\cap V(Q_d)$ if and only if $a(v)\in H^+\cap V(Q_d)$.
    
    In all cases, the assertion follows.
\end{proof}

\begin{lemma}\label{Lemma:HypercubeFirstCuts}
Let $Q_d$ the $d$-dimensional hypercube, and $H\in \mathcal{H}_{Q_d}$, with $H^-$ being the open half-space with fewer vertices of $Q_d$. Then, the following table holds.

\begin{table}[ht]
    \centering
    \begin{tabular}{|c|c|}
    \hline
         Value of $\vert H^-\cap V(Q_d)\vert$ &  Possible values of $cv_{Q_d}(H)$\\
         \hline
         0 & $2^r$, with $0\leq r\leq d-1$\\
         \hline
         1 & $d$\\
         \hline
         2 & $2d-2$\\
         \hline
         3 & $3d-5$ and $3d-4$\\
         \hline
         4 & $4d-i$, for $i=6, 7, 8, 9.$\\
        \hline
    \end{tabular}
    \caption{Some values of $cv_{Q_d}(H)$}
    \label{tab:my_label}
\end{table}
\end{lemma}

\begin{proof}

For $\vert H^-\cap V(Q_d)\vert=0$, we have that $H$ is a support hyperplane, so it is well known that $H\cap Q_d$
is a hypercube $Q_r$, with $0\leq r\leq d-1$, then $\vert V(H\cap Q_d)\vert=2^r$. 

By Proposition \ref{Proposition:Neighbors}, when $\vert H^-\cap V(Q_d)\vert=1, 2, 3$, the only configurations of vertices and edges in $H^-$ are the ones in Figure \ref{Fig:CubeMenos4}. For $\vert H^-\cap V(Q_d)\vert=4$ we may have the three cases that are shown in Figure \ref{Fig:Cube4}. The key idea in all the cases is to count the number of edges that are between $H^-$ and $H\cup H^+$, and how many of them might count only as one vertex in case $H$ only intersect them in their final points. These would be the possible values of $cv_{Q_d}(H)$ in each case.

When $\vert H^-\cap V(Q_d)\vert=1$, let $v\in H^-\cap V(Q_d)$ as in Figure \ref{Fig:CubeMenos4} a). Each of the $d$ edges adjacent to $v$ must intersect $H$ in $d$ different points, which are actually the vertices of $H\cap Q_d$.

When $\vert H^-\cap V(Q_d)\vert=2$, let $v_1,v_2 \in V(Q_d)\cap H^-$ as in Figure \ref{Fig:CubeMenos4} b). Since $N(v_1)\cap N(v_2)=\emptyset$, otherwise the $1-$skeleton of $Q_d$ would have a triangle, all the edges adjacent to $v_1$ and $v_2$, except $\{v_1, v_2\}$, must intersect $H$ only in one point each. Then $H\cap Q_d$ has $2(d-1)$ vertices.

When $\vert H^-\cap V(Q_d)\vert=3$, let $v_1,v_2,v_3 \in V(Q_d)\cap H^-$ as in Figure \ref{Fig:CubeMenos4} c). Note that $N(v_1)\cap N(v_2)=\emptyset$, $N(v_3)\cap N(v_2)=\emptyset$, but $N(v_1)\cap N(v_1)=v_4$, for some $v_4\in H\cup H^+$. If $v_4\in H^+$, then the adjacent edges to $v_1, v_2, v_3$, except the ones totally contained in $H^-$, must intersect $H$ only in one point each. Then, $H\cap Q_d$ would have $3d-4$ vertices. If $v_4\in H$, the edges $\{v_1, v_4\}$ and $\{v_3, v_4\}$ would count just by one vertex, then $H\cap Q_d$ would have $3d-5$ vertices.

When $\vert H^-\cap V(Q_d)\vert=4$, we have three cases. Let $v_1,v_2,v_3,v_4 \in V(Q_d)\cap H^-$ as in Figure \ref{Fig:Cube4} a). Note that $N(v_1)\cap N(v_2)=\emptyset$, $N(v_2)\cap N(v_3)=\emptyset$, $N(v_3)\cap N(v_4)=\emptyset$, $N(v_1)\cap N(v_4)=\emptyset$, because the $1-$skeleton of $Q_d$ has not odd cycles. But $N(v_1)\cap N(v_3)=v_5$ and $N(v_2)\cap N(v_4)=v_6$, for some different vertices $v_5, v_6\in V(Q_d)$. Following the same analysis as above and depending on whether $v_5$ and $v_6$ are in $H$ or $H^+$ we have that the possible number of vertices of $H\cap Q_d$ would be $4d-6, 4d-7$ or $4d-8$ vertices.

Now, let $v_1,v_2,v_3,v_4 \in V(Q_d)\cap H^-$ as in Figure \ref{Fig:Cube4} b). Note that $N(v_1)\cap N(v_2)=\emptyset$, $N(v_1)\cap N(v_3)=\emptyset$, $N(v_1)\cap N(v_4)=\emptyset$, because the $1-$skeleton of $Q_d$ has not odd cycles. But $N(v_2)\cap N(v_3)=v_5$, $N(v_2)\cap N(v_4)=v_6$, and $N(v_3)\cap N(v_4)=v_7$ for some different $v_5, v_6, v_7\in V(Q_d)$. Following the same analysis as above and depending on whether $v_5$, $v_6$, and $v_7$ are in $H$ or $H^+$ we have that the possible number of vertices of $H\cap Q_d$ would be $4d-6, 4d-7$, $4d-8$ or $4d-9$ vertices.

Finally, let $v_1,v_2,v_3,v_4 \in V(Q_d)\cap H^-$ as in Figure \ref{Fig:Cube4} c). Note that $N(v_i)\cap N(v_j)=\emptyset$, for all $i,j=1,2,3,4$, with $i\neq j$. Then, all the edges adjacent to $v_i$ and $v_j$, except the ones $\{v_i, v_j\}$, must intersect $H$ only in one point each. Then $H\cap Q_d$ has $4(d-2)$ vertices.
\end{proof}

\begin{figure}[ht]
\centering
\input{Figures/grados_cubo1.tikz}
\caption{Configurations when $\vert H^-\cap V(Q_d)\vert <4$}
\label{Fig:CubeMenos4}
\end{figure}

\begin{figure}[ht]
\centering
\input{Figures/grados_cubo2.tikz}
\caption{Configurations when $\vert H^-\cap V(Q_d)\vert =4$}
\label{Fig:Cube4}
\end{figure}

\begin{lemma}\label{Lemma:AtLeast4d-9}
    Let $Q_d$ the $d$-dimensional hypercube, with $d\geq 4$, and $H\in \mathcal{H}_{Q_d}$. If each half-space induced by $H$ has at least $4$ vertices of $Q_d$, then $cv_{Q_d}(H)\geq 4d-9$.
\end{lemma}

\begin{proof}
The key idea of this proof is to find pairwise disjoint sub-hypercubes $Q^1_{d_1},\dots, Q^r_{d_r}$ contained in $Q_d$, with $1\leq d_1\leq \dots \leq d_r<d$ and some positive integer $r$, such that each of them contains at least one pair of antipodal vertices (antipodal restricted to each of these sub-hypercubes, not antipodal in $Q_d$) in different half-spaces. This would imply, by Menger's Theorem (for more details see \cite{bondy2008graph}, that $$cv_{Q_d}(H)\geq \sum_{i=1}^r d_i,$$ since $Q_{d_i}$ has $d_i$ disjoint paths between any pair of antipodal vertices. 

Assume $\vert H^-\cap V(Q_d)\vert \leq \vert H^+\cap V(Q_d)\vert$. We begin with the following claim: if $H^-$ contains a square in the 1-skeleton of $Q_d$, then $cv_{Q_d}(H)\geq 4d-8$.

Indeed, without loss of generality (by relabeling) suppose the vertices $(0,0,\dots,0),$ $ (1,0,\dots,0),$ $ (0,1,0,\dots,0), $ $(1,1,0,\dots,0)\in H^-.$ 

Then, by Lemma \ref{Lemma:NoAntipodal}, their antipodals points $(1,1,\dots,1),$ $ (0,1,\dots,1),$ $ (1,0,1,\dots,1), $ $(0,0,1,\dots,1)\in H^+.$ Let $Q^{ij}_{d-2}$ be the sub-hypercubes of $Q_d$ with the first entry fixed as $i$ and the second entry as $j$, for $i,j=0,1$. These four sub-hypercubes are pairwise disjoint, and 

$$(0,0,\dots,0), (0,0,1,\dots,1)\in Q^{00}_{d-2}, $$
$$(1,0,\dots,0), (1,0,1,\dots,1)\in Q^{10}_{d-2}, $$
    $$(0,1,0,\dots,0), (0,1,\dots,1)\in Q^{01}_{d-2}, $$
$$(1,1,0,\dots,0), (1,1,\dots,1)\in Q^{11}_{d-2},$$

which are antipodal pairs of points in each hypercube. Therefore, $cv_{Q_d}(H)\geq 4(d-2)$, proving the claim.

Now, we proceed by cases:

\textbf{Case 1:} Suppose that $u$ has no zero entries. By Corollary \ref{Corollary:NonZero}, Proposition \ref{Proposition:equivalenceQd}, and Theorem \ref{Thm:antichain}, there exist a $u'\in \mathbb{S}^{d-1}$ with all positive entries and a hyperplane $H':\{X: X\cdot u'=t'\}$, for some $t'\in\mathbb{R}$, such that the maximal antichain in $(Q_u^*,\leq)$ induced by $H$ is equivalent to the maximal antichain in $(Q_u'^*,\leq)$ induced by $H'$.

Since $H'^-$ contains at least four vertices, we have two sub-cases:

\textbf{Sub-case 1.1:} Suppose that $H^-$ contains a vertex of $Q_d$ with at least two positive entries. Then, by the structure of $(Q_u'^*,\leq)$, there is a $a(v)\in H'^-\cap V(Q_d)$ with exactly two positive entries. Without loss of generality (by relabeling), suppose $a(v)=(1,1, 0, \cdots,0)$. Then $(0,0,\dots,0)$, $(1,0,\dots,0)$, $(0,1,0,\dots,0)\in H'^-$. 

Thus, $H^-$ contains a square, and form the earlier claim, it follows that $cv_{Q_d}(H)\geq 4d-8$. 

\textbf{Sub-case 1.2:} Suppose all vertices of $Q_d$ in $H'^-$ have at most two positive entries. Then the vertex $(0,0,\dots,0)$ and at least three vertices with a single positive entry lie in $H^-$. Without loss of generality (by relabeling), assume $(1,0,\dots,0)$, $(0,1,0,\dots,0), (0,0,1,0,\dots,0)\in H'^-$.

We define the following pairwise disjoint sub-hypercubes: 

\begin{itemize}
    \item $Q^1_{d-1}$, whose vertices have the first entry equal to 1.
    \item $Q^2_{d-2}$, whose vertices have the first entry equal to 0 and the second equal to 1.
    \item $Q^3_{d-3}$, whose vertices have the first and second entries equal to 0 and the third equal to 1.
    \item $Q^4_{d-3}$, whose vertices have the first, second, and third entries equal to 0.
\end{itemize}

Each of the four points above lies in one of these subcubes, and their respective antipodal vertices (restricted to each sub-hypercube) must lie in $H^+$, since $H'$ only contains vertices with at most two positive entries. Therefore, using the key idea, we obtain $cv_{Q_d}(H')=cv_{Q_d}(H)\geq 4d-9$.

\textbf{Case 2:} For this and the next case, we use the notation of Proposition \ref{Pro:ZeroEntries}. Suppose that $u$ has exactly one zero entry. Let $Q^1_{d-1}$ and $Q^2_{d-1}$ be the sub-hypercubes as in Proposition \ref{Pro:ZeroEntries}. Then $\vert V(Q^1_{d-1})\cap H^-\vert\geq 2$, which means that $H^-$ contains at least one edge of $Q^1_{d-1}$. Let $\{v_1, v_2\}$ be such an edge.

Then, $\{w^2_1, w^2_2\}$ is an edge, with $w^2_1= V(Q^{2}_{d-1})\cap V_{v_1}$ and $w_2= V(Q^{2}_{d-1})\cap V_{v_2}$ (see item 4. in the proof of Proposition \ref{Pro:ZeroEntries}). Furthermore,  since $v_i$ and $w_i^2$ only have the entry different where $u$ has the zero entry, for $i=1,2$, the four vertices $v_1, v_2, w^2_1$ and $w^2_2$ define a square in $H^-$. Therefore, $cv_{Q_d}(H)\geq 4d-8$.

\textbf{Case 3:}  Suppose that  $u$ has $k$ entries equal to zero entries, with $k\geq 2$. Let $Q^r_{d-1}$, for $r=1,2,3,4$, be the sub-hypercubes as in Proposition \ref{Pro:ZeroEntries}. Since $\vert V(Q^1_{d-k})\cap H^-\vert\geq 1$, there is a vertex $v\in V(Q^1_{d-k})\cap H^-$.

For each $r=1,2,3,4$, define $v_r\in V_{v}\cap Q^r_{d-1}$ (note that $v=v_1$), then $v_r\in H^-$, for all $r$, and the four vertices $v_1, v_2, v_3, v_4$ form a square in the $1-$skeleton of $Q_d$. Therefore, it follows that $cv_{Q_d}(H)\geq 4d-8$. 

\end{proof}

\begin{proof}[ Proof of Theorem \ref{Thm:FirstGAPS}]
    By Lemmas \ref{Lemma:HypercubeFirstCuts} and \ref{Lemma:AtLeast4d-9}, $m\leq 4d-10$ is a gap if $m$ is not in $I=\{2^{i}, \text{ for }0\leq i\leq d-1\}\cup\{d, 2d-2, 3d-5, 3d-4\}$.
\end{proof}



Regarding Theorem \ref{Thm:gap_penultimo}, recall that the size of the largest antichain in a poset is called \emph{the width}. By Lemma \ref{lemma 1} and O'Neil's result in \cite{o1971hyperplane}, the width of $(Q_{d_(1,\dots, 1)}, \leq)$ is $\lceil d/2 \rceil \binom{d}{\lfloor d/2 \rfloor}$, furthermore, when $u$ has entries equal to zero, there are no slices with $\lceil d/2 \rceil \binom{d}{\lfloor d/2 \rfloor}$ vertices. In the following lemma, we will show the structure of the largest antichains in the basic poset of the hypercube.

\begin{lemma}\label{lemma:central_edges}
The width of $(Q_{d{(1,\dots, 1)}}, \leq)$ is $\lceil d/2 \rceil \binom{d}{\lfloor d/2 \rfloor}$. Moreover, if $A \subset (Q_{d(1,\dots, 1)}, \leq)$ is an antichain of size equal to the width, then $A$ only contains edges located at the central levels. That is, $e \in A$ if and only if $e \in E_{\lceil d/2 \rceil} \cup E_{\lfloor d/2 \rfloor}$, where
\[
E_i = \left\{e = (v_1, v_2) \in E(Q_d) : \|v_1\| + 1 = \|v_2\| = i\right\}, \quad i = 1, 2, \dots, d.
\]
\end{lemma}

\begin{proof}
    In \cite{o1971hyperplane}, O'Neil proved that the slice intersecting the largest possible number of edges in $Q_d$ crosses exactly $\lceil d/2 \rceil \binom{d}{\lfloor d/2 \rfloor}$ edges, by exhibiting a slice that only intersects edges in the central levels. This set of edges forms an antichain in our poset. Moreover, from Lemma~\ref{lemma 1} and Theorem~\ref{Thm:antichain}, we deduce that this is the maximum possible size for an antichain. It remains to show that any other maximum-size antichain must be of this form; that is, it does not contain vertices of $Q_d$ and only contains edges from the central levels.

    We define the vertex levels as follows:
    \[
    V_i = \left\{ v \in V(Q_d) : \|v\| = i \right\}, \quad i = 0, 1, \dots, d.
    \]

    Note that if $v$ is a vertex at level $i$, then there are exactly $(d - i)! \, i!$ distinct paths from $0$ to $1$ passing through $v$, while if $e$ is an edge of $Q_d$ at level $i$, there are exactly $(d - i)! \, (i - 1)!$ such paths passing through $e$. Observe that, for edges, the number of paths passing through them is minimized when $i = \lceil d/2 \rceil$ if $d$ is odd, and when $i \in \{d/2, d/2 - 1\}$ if $d$ is even; that is, when the edges are in the central levels. Moreover, any vertex, regardless of its level, lies on more paths than an edge in the central levels.

    Now, aiming for a contradiction with the total number of paths from $0$ to $1$, suppose there exists an antichain $A$ of size $\lceil d/2 \rceil \binom{d}{\lfloor d/2 \rfloor}$ that contains vertices or non-central edges. Since no two elements in an antichain can lie on the same path from $0$ to $1$, the sets of paths passing through each element of $A$ must be disjoint. Then, by counting the total number of distinct paths from $0$ to $1$ passing through elements of $A$, we get a number strictly greater than:
    \[
    \lceil d/2 \rceil \binom{d}{\lfloor d/2 \rfloor} (\lfloor d/2 \rfloor)!^2 = d!, \quad \text{if $d$ is odd},
    \]
    \[
    \lceil d/2 \rceil \binom{d}{\lfloor d/2 \rfloor} (d/2)! \, (d/2 - 1)! = d!, \quad \text{if $d$ is even}.
    \]
    
    which is a contradiction, since the total number of distinct paths from $0$ to $1$ is exactly $d!$.
\end{proof}


Now we proceed to prove Theorem \ref{Thm:gap_penultimo}.

\begin{proof} [Proof of Theorem \ref{Thm:gap_penultimo}]
Suppose there is a hyperplane $ H $ that induces a slice with exactly 
$\frac{d}{2} \binom{d}{d/2} - 1$ vertices. Since $ d $ is even, this number is odd because:
\[
\frac{d}{2} \binom{d}{d/2} - 1 = \frac{d}{2} \cdot 2 \binom{d-1}{d/2-1} - 1 = d \binom{d-1}{d/2-1} - 1.
\]

By Proposition \ref{prop:even_cut}, this hyperplane $ H $ must contain at least one vertex of $ Q_d $. Furthermore, by Lemma \ref{lemma 1}, we can continuously adjust $ H $ to obtain a new hyperplane $ H' $ that does not contain any vertices of $ Q_d $, while ensuring that the slice $ H' \cap Q_d $ contains at least as many vertices as $ H \cap Q_d $. In this case, $ H' \cap Q_d $ must achieve exactly $ \frac{d}{2} \binom{d}{d/2} $ vertices, which implies that $ H' $ induces a maximal antichain in the poset $ (P_u, \leq) $ for some $ u $.

After maybe a relabeling of the vertices of $Q_d$ by \ref{lemma:central_edges}, we can assume that the slice determined by $ H'$ intersects edges in the central levels. This implies that by moving $ H' $ in the opposite direction to the adjustment made to $ H $, at some point, two edges will converge at a vertex. Observe that for each vertex $ u $, such that $ e = \{u, v\} \in E_i $, $ u $ is incident to exactly $ d-i $ edges in $ E_i $ and $ i $ edges in $ E_{i-1} $.

This observation implies that for $ d > 4 $, vertices that are endpoints of edges in the central levels have more than two edges in each of the consecutive levels they belong to. Consequently, moving $ H' $ back toward $ H $ does not decrease the number of vertices induced by the slice by one. Therefore, such a configuration is not possible.
\end{proof}

\subsection{The VSS of low dimensional hypercubes}

For several years, the goal of obtaining all the combinatorial types of slices of the hypercube has been a significant computational problem. For dimensions 2 and 3, all combinatorial slices are known, which are relatively easy to compute, and for dimension 4 and 5 all the slices have been characterized in \cite{fukuda1997sections}. Although we do not present the combinatorial types of slices of $Q_6$ and $Q_7$, we present their exact vertex sequences of slices, completing Table \ref{tab:VSS} up to dimension 7. We also know that Marie-Charlotte Brandenburg and Chiara Meroni are currently working on the explicit enumeration of all combinatorially distinct slices of $Q_6$.

We prove Corollary \ref{Corollary:TableUpto7}

\begin{proof}
The VSS of the hypercubes up to dimension 5 are already compute as we said above. Then, we only need to compute the VSS of the hypercubes of dimensions 6 and 7.

For dimension 6, we apply Theorem \ref{Thm:FirstGAPS} and Theorem \ref{Thm:gap_penultimo} to prove that $3, 5, 7, 9, 11, 12$ and $59$ are gaps. For dimension 7, we apply only Theorem \ref{Thm:FirstGAPS} to prove that $3, 5, 6, 9, 10, 11, 13, 14, 15, 18$ are gaps.

We recall that $\nu(Q_6)=60$ and $\nu(Q_7)=140$, as shown in \cite{o1971hyperplane}. To prove that these are the only missing values, we explicitly construct slices realizing all other values listed in Table~1 using the code available at \href{https://github.com/AntonioJTH/slicing_polytopes}{\texttt{https://github.com/AntonioJTH/slicing\_polytopes}}. This Python program computes slices of the hypercube along selected directions. For each direction, the program generates all possible slices and records the number of vertices in each. The output consists of all values of $cv_P$ realized by some hyperplane orthogonal to one of the given directions.
\end{proof}

{\bf Acknowledgments:} We thank Anouk Brose, Chiara Meroni,  and Marie Brandenburg for their comments. This research was partially supported by NSF Grant DMS-2348578, NSF Grant DMS-2434665, and NSF Grant DMS-1929284 of ICERM.
\bibliographystyle{plainurl}
\bibliography{main.bib}

\appendix

\end{document}

%% file: Figures/cyclic.tikz
\tikzset{every picture/.style={line width=0.75pt}} 

\begin{tikzpicture}[x=0.75pt,y=0.75pt,yscale=-.9,xscale=.9]

\draw  [draw opacity=0][fill={rgb, 255:red, 255; green, 255; blue, 255 }  ,fill opacity=1 ] (504.09,76.01) .. controls (533.27,112.27) and (533.39,163.63) .. (502.47,196.2) .. controls (471.25,229.07) and (419.27,231.34) .. (381.45,203.24) -- (433.4,130.61) -- cycle ; \draw  [color={rgb, 255:red, 128; green, 128; blue, 128 }  ,draw opacity=0.3 ] (504.09,76.01) .. controls (533.27,112.27) and (533.39,163.63) .. (502.47,196.2) .. controls (471.25,229.07) and (419.27,231.34) .. (381.45,203.24) ;  
\draw [color={rgb, 255:red, 128; green, 128; blue, 128 }  ,draw opacity=1 ]   (432.47,221.73) -- (522.32,161.39) ;
\draw [color={rgb, 255:red, 128; green, 128; blue, 128 }  ,draw opacity=1 ][fill={rgb, 255:red, 155; green, 155; blue, 155 }  ,fill opacity=1 ]   (465.35,219.18) -- (513.46,89.64) ;
\draw [color={rgb, 255:red, 128; green, 128; blue, 128 }  ,draw opacity=1 ][fill={rgb, 255:red, 155; green, 155; blue, 155 }  ,fill opacity=1 ]   (497.68,200.68) -- (513.46,89.64) ;
\draw [color={rgb, 255:red, 128; green, 128; blue, 128 }  ,draw opacity=1 ][fill={rgb, 255:red, 155; green, 155; blue, 155 }  ,fill opacity=1 ]   (432.47,221.73) -- (513.46,89.64) ;
\draw [color={rgb, 255:red, 128; green, 128; blue, 128 }  ,draw opacity=1 ][fill={rgb, 255:red, 155; green, 155; blue, 155 }  ,fill opacity=1 ]   (396.29,212.03) -- (522.32,161.39) ;
\draw [color={rgb, 255:red, 128; green, 128; blue, 128 }  ,draw opacity=1 ][fill={rgb, 255:red, 155; green, 155; blue, 155 }  ,fill opacity=1 ]   (396.29,212.03) -- (524.94,124.75) ;
\draw [color={rgb, 255:red, 128; green, 128; blue, 128 }  ,draw opacity=1 ][fill={rgb, 255:red, 155; green, 155; blue, 155 }  ,fill opacity=1 ]   (497.68,200.68) -- (522.32,161.39) ;
\draw [color={rgb, 255:red, 128; green, 128; blue, 128 }  ,draw opacity=1 ][fill={rgb, 255:red, 155; green, 155; blue, 155 }  ,fill opacity=1 ]   (396.29,212.03) -- (513.46,89.64) ;
\draw [color={rgb, 255:red, 128; green, 128; blue, 128 }  ,draw opacity=1 ][fill={rgb, 255:red, 155; green, 155; blue, 155 }  ,fill opacity=1 ]   (524.94,124.75) -- (432.47,221.73) ;
\draw [color={rgb, 255:red, 128; green, 128; blue, 128 }  ,draw opacity=1 ][fill={rgb, 255:red, 155; green, 155; blue, 155 }  ,fill opacity=1 ]   (522.32,161.39) -- (465.35,219.18) ;
\draw  [draw opacity=0][fill={rgb, 255:red, 255; green, 255; blue, 255 }  ,fill opacity=1 ] (237.82,70.48) .. controls (265.83,107.66) and (264.31,159) .. (232.36,190.56) .. controls (200.11,222.42) and (148.08,223.03) .. (111.18,193.74) -- (165.42,122.8) -- cycle ; \draw  [color={rgb, 255:red, 128; green, 128; blue, 128 }  ,draw opacity=0.3 ] (237.82,70.48) .. controls (265.83,107.66) and (264.31,159) .. (232.36,190.56) .. controls (200.11,222.42) and (148.08,223.03) .. (111.18,193.74) ;  
\draw [color={rgb, 255:red, 128; green, 128; blue, 128 }  ,draw opacity=1 ][fill={rgb, 255:red, 155; green, 155; blue, 155 }  ,fill opacity=1 ]   (125.72,202.99) -- (194.53,212.35) ;
\draw [color={rgb, 255:red, 128; green, 128; blue, 128 }  ,draw opacity=1 ][fill={rgb, 255:red, 155; green, 155; blue, 155 }  ,fill opacity=1 ]   (161.58,213.85) -- (194.53,212.35) ;
\draw [color={rgb, 255:red, 128; green, 128; blue, 128 }  ,draw opacity=1 ][fill={rgb, 255:red, 155; green, 155; blue, 155 }  ,fill opacity=1 ]   (125.72,202.99) -- (227.43,194.89) ;
\draw [color={rgb, 255:red, 128; green, 128; blue, 128 }  ,draw opacity=1 ][fill={rgb, 255:red, 155; green, 155; blue, 155 }  ,fill opacity=1 ]   (253.31,156.4) -- (246.74,84.41) ;
\draw [color={rgb, 255:red, 128; green, 128; blue, 128 }  ,draw opacity=1 ][fill={rgb, 255:red, 155; green, 155; blue, 155 }  ,fill opacity=1 ]   (194.53,212.35) -- (246.74,84.41) ;
\draw [color={rgb, 255:red, 128; green, 128; blue, 128 }  ,draw opacity=1 ][fill={rgb, 255:red, 155; green, 155; blue, 155 }  ,fill opacity=1 ]   (227.43,194.89) -- (246.74,84.41) ;
\draw [color={rgb, 255:red, 128; green, 128; blue, 128 }  ,draw opacity=1 ][fill={rgb, 255:red, 155; green, 155; blue, 155 }  ,fill opacity=1 ]   (257.1,119.87) -- (246.74,84.41) ;
\draw [color={rgb, 255:red, 128; green, 128; blue, 128 }  ,draw opacity=1 ][fill={rgb, 255:red, 155; green, 155; blue, 155 }  ,fill opacity=1 ]   (253.31,156.4) -- (257.1,119.87) ;
\draw [color={rgb, 255:red, 128; green, 128; blue, 128 }  ,draw opacity=1 ][fill={rgb, 255:red, 155; green, 155; blue, 155 }  ,fill opacity=1 ]   (125.72,202.99) -- (161.58,213.85) ;
\draw [color={rgb, 255:red, 128; green, 128; blue, 128 }  ,draw opacity=1 ][fill={rgb, 255:red, 155; green, 155; blue, 155 }  ,fill opacity=1 ]   (161.58,213.85) -- (246.74,84.41) ;
\draw [color={rgb, 255:red, 128; green, 128; blue, 128 }  ,draw opacity=1 ][fill={rgb, 255:red, 155; green, 155; blue, 155 }  ,fill opacity=1 ]   (194.53,212.35) -- (227.43,194.89) ;
\draw [color={rgb, 255:red, 128; green, 128; blue, 128 }  ,draw opacity=1 ][fill={rgb, 255:red, 155; green, 155; blue, 155 }  ,fill opacity=1 ]   (125.72,202.99) -- (253.31,156.4) ;
\draw [color={rgb, 255:red, 128; green, 128; blue, 128 }  ,draw opacity=1 ][fill={rgb, 255:red, 155; green, 155; blue, 155 }  ,fill opacity=1 ]   (125.72,202.99) -- (257.1,119.87) ;
\draw [color={rgb, 255:red, 128; green, 128; blue, 128 }  ,draw opacity=1 ][fill={rgb, 255:red, 155; green, 155; blue, 155 }  ,fill opacity=1 ]   (227.43,194.89) -- (253.31,156.4) ;
\draw [color={rgb, 255:red, 128; green, 128; blue, 128 }  ,draw opacity=1 ][fill={rgb, 255:red, 155; green, 155; blue, 155 }  ,fill opacity=1 ]   (125.72,202.99) -- (246.74,84.41) ;
\draw  [color={rgb, 255:red, 128; green, 128; blue, 128 }  ,draw opacity=1 ][fill={rgb, 255:red, 155; green, 155; blue, 155 }  ,fill opacity=1 ] (124.64,205.03) .. controls (123.45,204.45) and (122.97,203.07) .. (123.56,201.94) .. controls (124.16,200.82) and (125.61,200.38) .. (126.8,200.96) .. controls (128,201.54) and (128.48,202.92) .. (127.89,204.04) .. controls (127.29,205.16) and (125.84,205.6) .. (124.64,205.03) -- cycle ;
\draw  [color={rgb, 255:red, 128; green, 128; blue, 128 }  ,draw opacity=1 ][fill={rgb, 255:red, 155; green, 155; blue, 155 }  ,fill opacity=1 ] (160.5,215.88) .. controls (159.31,215.3) and (158.82,213.92) .. (159.42,212.8) .. controls (160.02,211.68) and (161.47,211.24) .. (162.66,211.82) .. controls (163.86,212.39) and (164.34,213.77) .. (163.74,214.9) .. controls (163.15,216.02) and (161.7,216.46) .. (160.5,215.88) -- cycle ;
\draw  [color={rgb, 255:red, 128; green, 128; blue, 128 }  ,draw opacity=1 ][fill={rgb, 255:red, 155; green, 155; blue, 155 }  ,fill opacity=1 ] (193.45,214.39) .. controls (192.26,213.81) and (191.77,212.43) .. (192.37,211.3) .. controls (192.96,210.18) and (194.41,209.74) .. (195.61,210.32) .. controls (196.8,210.9) and (197.29,212.28) .. (196.69,213.4) .. controls (196.09,214.52) and (194.64,214.96) .. (193.45,214.39) -- cycle ;
\draw  [color={rgb, 255:red, 128; green, 128; blue, 128 }  ,draw opacity=1 ][fill={rgb, 255:red, 155; green, 155; blue, 155 }  ,fill opacity=1 ] (226.35,196.93) .. controls (225.16,196.35) and (224.67,194.97) .. (225.27,193.85) .. controls (225.86,192.72) and (227.32,192.28) .. (228.51,192.86) .. controls (229.7,193.44) and (230.19,194.82) .. (229.59,195.94) .. controls (229,197.07) and (227.54,197.51) .. (226.35,196.93) -- cycle ;
\draw  [color={rgb, 255:red, 128; green, 128; blue, 128 }  ,draw opacity=1 ][fill={rgb, 255:red, 155; green, 155; blue, 155 }  ,fill opacity=1 ] (252.23,158.44) .. controls (251.04,157.86) and (250.55,156.48) .. (251.15,155.36) .. controls (251.75,154.23) and (253.2,153.79) .. (254.39,154.37) .. controls (255.59,154.95) and (256.07,156.33) .. (255.47,157.45) .. controls (254.88,158.58) and (253.43,159.02) .. (252.23,158.44) -- cycle ;
\draw  [color={rgb, 255:red, 128; green, 128; blue, 128 }  ,draw opacity=1 ][fill={rgb, 255:red, 155; green, 155; blue, 155 }  ,fill opacity=1 ] (256.02,121.91) .. controls (254.83,121.33) and (254.35,119.95) .. (254.94,118.82) .. controls (255.54,117.7) and (256.99,117.26) .. (258.18,117.84) .. controls (259.38,118.42) and (259.86,119.8) .. (259.27,120.92) .. controls (258.67,122.04) and (257.22,122.48) .. (256.02,121.91) -- cycle ;
\draw  [color={rgb, 255:red, 128; green, 128; blue, 128 }  ,draw opacity=1 ][fill={rgb, 255:red, 155; green, 155; blue, 155 }  ,fill opacity=1 ] (245.66,86.45) .. controls (244.47,85.87) and (243.98,84.49) .. (244.58,83.37) .. controls (245.18,82.24) and (246.63,81.8) .. (247.82,82.38) .. controls (249.02,82.96) and (249.5,84.34) .. (248.9,85.46) .. controls (248.31,86.59) and (246.86,87.03) .. (245.66,86.45) -- cycle ;
\draw [color={rgb, 255:red, 128; green, 128; blue, 128 }  ,draw opacity=1 ][fill={rgb, 255:red, 155; green, 155; blue, 155 }  ,fill opacity=1 ]   (522.32,161.39) -- (513.46,89.64) ;
\draw [color={rgb, 255:red, 128; green, 128; blue, 128 }  ,draw opacity=1 ][fill={rgb, 255:red, 155; green, 155; blue, 155 }  ,fill opacity=1 ]   (524.94,124.75) -- (513.46,89.64) ;
\draw [color={rgb, 255:red, 128; green, 128; blue, 128 }  ,draw opacity=1 ][fill={rgb, 255:red, 155; green, 155; blue, 155 }  ,fill opacity=1 ]   (522.32,161.39) -- (524.94,124.75) ;
\draw [color={rgb, 255:red, 128; green, 128; blue, 128 }  ,draw opacity=1 ][fill={rgb, 255:red, 155; green, 155; blue, 155 }  ,fill opacity=1 ]   (396.29,212.03) -- (465.35,219.18) ;
\draw [color={rgb, 255:red, 128; green, 128; blue, 128 }  ,draw opacity=1 ][fill={rgb, 255:red, 155; green, 155; blue, 155 }  ,fill opacity=1 ]   (396.29,212.03) -- (432.47,221.73) ;
\draw [color={rgb, 255:red, 128; green, 128; blue, 128 }  ,draw opacity=1 ][fill={rgb, 255:red, 155; green, 155; blue, 155 }  ,fill opacity=1 ]   (432.47,221.73) -- (465.35,219.18) ;
\draw [color={rgb, 255:red, 128; green, 128; blue, 128 }  ,draw opacity=1 ][fill={rgb, 255:red, 155; green, 155; blue, 155 }  ,fill opacity=1 ]   (465.35,219.18) -- (497.68,200.68) ;
\draw [color={rgb, 255:red, 128; green, 128; blue, 128 }  ,draw opacity=1 ][fill={rgb, 255:red, 155; green, 155; blue, 155 }  ,fill opacity=1 ]   (396.29,212.03) -- (497.68,200.68) ;
\draw  [color={rgb, 255:red, 128; green, 128; blue, 128 }  ,draw opacity=1 ][fill={rgb, 255:red, 155; green, 155; blue, 155 }  ,fill opacity=1 ] (395.27,214.09) .. controls (394.06,213.55) and (393.53,212.19) .. (394.09,211.05) .. controls (394.65,209.91) and (396.09,209.42) .. (397.3,209.96) .. controls (398.51,210.5) and (399.04,211.86) .. (398.48,213) .. controls (397.92,214.15) and (396.48,214.63) .. (395.27,214.09) -- cycle ;
\draw  [color={rgb, 255:red, 128; green, 128; blue, 128 }  ,draw opacity=1 ][fill={rgb, 255:red, 155; green, 155; blue, 155 }  ,fill opacity=1 ] (521.31,163.45) .. controls (520.1,162.91) and (519.57,161.55) .. (520.13,160.41) .. controls (520.69,159.27) and (522.12,158.78) .. (523.34,159.32) .. controls (524.55,159.86) and (525.08,161.23) .. (524.52,162.37) .. controls (523.96,163.51) and (522.52,163.99) .. (521.31,163.45) -- cycle ;
\draw  [color={rgb, 255:red, 128; green, 128; blue, 128 }  ,draw opacity=1 ][fill={rgb, 255:red, 155; green, 155; blue, 155 }  ,fill opacity=1 ] (512.44,91.71) .. controls (511.23,91.17) and (510.7,89.81) .. (511.26,88.66) .. controls (511.82,87.52) and (513.26,87.04) .. (514.47,87.58) .. controls (515.68,88.12) and (516.21,89.48) .. (515.65,90.62) .. controls (515.09,91.76) and (513.65,92.25) .. (512.44,91.71) -- cycle ;
\draw [color={rgb, 255:red, 144; green, 19; blue, 254 }  ,draw opacity=1 ][fill={rgb, 255:red, 155; green, 155; blue, 155 }  ,fill opacity=1 ]   (266.96,83.52) -- (120.56,225.52) ;
\draw [color={rgb, 255:red, 144; green, 19; blue, 254 }  ,draw opacity=1 ][fill={rgb, 255:red, 155; green, 155; blue, 155 }  ,fill opacity=1 ]   (532.96,193.79) -- (443.36,127.79) ;
\draw [color={rgb, 255:red, 128; green, 128; blue, 128 }  ,draw opacity=1 ]   (497.68,200.68) -- (432.47,221.73) ;
\draw [color={rgb, 255:red, 128; green, 128; blue, 128 }  ,draw opacity=1 ]   (465.35,219.18) -- (524.94,124.75) ;
\draw  [color={rgb, 255:red, 128; green, 128; blue, 128 }  ,draw opacity=1 ][fill={rgb, 255:red, 155; green, 155; blue, 155 }  ,fill opacity=1 ] (431.46,223.8) .. controls (430.25,223.26) and (429.72,221.89) .. (430.28,220.75) .. controls (430.84,219.61) and (432.27,219.12) .. (433.49,219.67) .. controls (434.7,220.21) and (435.23,221.57) .. (434.67,222.71) .. controls (434.11,223.85) and (432.67,224.34) .. (431.46,223.8) -- cycle ;
\draw  [color={rgb, 255:red, 128; green, 128; blue, 128 }  ,draw opacity=1 ][fill={rgb, 255:red, 155; green, 155; blue, 155 }  ,fill opacity=1 ] (464.34,221.25) .. controls (463.13,220.71) and (462.6,219.35) .. (463.16,218.2) .. controls (463.72,217.06) and (465.16,216.58) .. (466.37,217.12) .. controls (467.58,217.66) and (468.11,219.02) .. (467.55,220.16) .. controls (466.99,221.3) and (465.55,221.79) .. (464.34,221.25) -- cycle ;
\draw [color={rgb, 255:red, 128; green, 128; blue, 128 }  ,draw opacity=1 ]   (524.94,124.75) -- (497.68,200.68) ;
\draw  [color={rgb, 255:red, 128; green, 128; blue, 128 }  ,draw opacity=1 ][fill={rgb, 255:red, 155; green, 155; blue, 155 }  ,fill opacity=1 ] (496.67,202.75) .. controls (495.46,202.21) and (494.93,200.85) .. (495.49,199.71) .. controls (496.05,198.56) and (497.48,198.08) .. (498.7,198.62) .. controls (499.91,199.16) and (500.44,200.52) .. (499.88,201.66) .. controls (499.32,202.8) and (497.88,203.29) .. (496.67,202.75) -- cycle ;
\draw  [color={rgb, 255:red, 128; green, 128; blue, 128 }  ,draw opacity=1 ][fill={rgb, 255:red, 155; green, 155; blue, 155 }  ,fill opacity=1 ] (523.93,126.82) .. controls (522.72,126.28) and (522.19,124.92) .. (522.75,123.77) .. controls (523.31,122.63) and (524.75,122.15) .. (525.96,122.69) .. controls (527.17,123.23) and (527.7,124.59) .. (527.14,125.73) .. controls (526.58,126.87) and (525.14,127.36) .. (523.93,126.82) -- cycle ;

\draw (271.92,67.56) node [anchor=north west][inner sep=0.75pt]    {$H$};
\draw (539.25,187.16) node [anchor=north west][inner sep=0.75pt]    {$H$};

\end{tikzpicture}

%% file: Figures/IcosahedronPaths.tikz
\tikzset{every picture/.style={line width=0.75pt}} 

\begin{tikzpicture}[x=0.75pt,y=0.75pt,yscale=-1,xscale=1]

\draw [color={rgb, 255:red, 183; green, 181; blue, 181 }  ,draw opacity=1 ]   (337.63,146.79) -- (336.94,175.2) ;
\draw [color={rgb, 255:red, 126; green, 211; blue, 33 }  ,draw opacity=1 ][line width=1.5]    (336.94,175.2) -- (337.63,146.79) ;
\draw [color={rgb, 255:red, 183; green, 181; blue, 181 }  ,draw opacity=1 ] [dash pattern={on 3.75pt off 1.5pt on 3.75pt off 1.5pt}]  (211.61,128.47) -- (232.41,141.28) ;
\draw [color={rgb, 255:red, 183; green, 181; blue, 181 }  ,draw opacity=1 ] [dash pattern={on 3.75pt off 1.5pt on 3.75pt off 1.5pt}]  (145.1,128.15) -- (211.61,128.47) ;
\draw [color={rgb, 255:red, 126; green, 211; blue, 33 }  ,draw opacity=1 ][line width=1.5]  [dash pattern={on 3.75pt off 1.5pt on 3.75pt off 1.5pt}]  (145.1,128.15) -- (120.8,140.2) ;
\draw [color={rgb, 255:red, 144; green, 19; blue, 254 }  ,draw opacity=1 ][line width=1.5]    (493.27,174.53) -- (493.96,146.12) ;
\draw [color={rgb, 255:red, 74; green, 144; blue, 226 }  ,draw opacity=1 ][line width=1.5]  [dash pattern={on 3.75pt off 1.5pt on 3.75pt off 1.5pt}]  (550.08,76.86) -- (528.61,129.47) ;
\draw  [color={rgb, 255:red, 183; green, 181; blue, 181 }  ,draw opacity=1 ] (549.41,142.28) -- (493.27,174.45) -- (437.79,141.12) -- (438.45,75.61) -- (494.59,43.44) -- (550.07,76.78) -- cycle ;
\draw [color={rgb, 255:red, 245; green, 166; blue, 35 }  ,draw opacity=1 ][line width=1.5]    (437.79,141.12) -- (493.27,174.45) ;
\draw [color={rgb, 255:red, 183; green, 181; blue, 181 }  ,draw opacity=1 ] [dash pattern={on 3.75pt off 1.5pt on 3.75pt off 1.5pt}]  (372.28,130.13) -- (336.94,175.2) ;
\draw [color={rgb, 255:red, 126; green, 211; blue, 33 }  ,draw opacity=1 ][line width=1.5]  [dash pattern={on 3.75pt off 1.5pt on 3.75pt off 1.5pt}]  (305.77,129.81) -- (336.94,175.2) ;
\draw  [color={rgb, 255:red, 183; green, 181; blue, 181 }  ,draw opacity=1 ][dash pattern={on 3.75pt off 1.5pt on 3.75pt off 1.5pt}] (338.23,146.63) -- (393.09,143.03) -- (372.28,130.13) -- (305.76,129.73) -- (281.46,141.78) -- cycle ;
\draw [color={rgb, 255:red, 183; green, 181; blue, 181 }  ,draw opacity=1 ]   (281.46,141.78) -- (337.63,146.79) ;
\draw [color={rgb, 255:red, 245; green, 166; blue, 35 }  ,draw opacity=1 ][line width=1.5]    (281.46,141.78) -- (337.63,146.79) ;
\draw [color={rgb, 255:red, 74; green, 144; blue, 226 }  ,draw opacity=1 ][line width=1.5]    (337.63,146.79) -- (393.08,143.03) ;
\draw [color={rgb, 255:red, 126; green, 211; blue, 33 }  ,draw opacity=1 ][line width=1.5]  [dash pattern={on 3.75pt off 1.5pt on 3.75pt off 1.5pt}]  (145.09,128.06) -- (177.04,72.3) ;
\draw [color={rgb, 255:red, 183; green, 181; blue, 181 }  ,draw opacity=1 ] [dash pattern={on 3.75pt off 1.5pt on 3.75pt off 1.5pt}]  (177.04,72.3) -- (211.61,128.47) ;
\draw [color={rgb, 255:red, 183; green, 181; blue, 181 }  ,draw opacity=1 ] [dash pattern={on 3.75pt off 1.5pt on 3.75pt off 1.5pt}]  (121.45,74.69) -- (145.09,128.06) ;
\draw [color={rgb, 255:red, 183; green, 181; blue, 181 }  ,draw opacity=1 ] [dash pattern={on 3.75pt off 1.5pt on 3.75pt off 1.5pt}]  (211.61,128.47) -- (176.27,173.53) ;
\draw [color={rgb, 255:red, 183; green, 181; blue, 181 }  ,draw opacity=1 ] [dash pattern={on 3.75pt off 1.5pt on 3.75pt off 1.5pt}]  (145.09,128.06) -- (176.27,173.53) ;
\draw [color={rgb, 255:red, 183; green, 181; blue, 181 }  ,draw opacity=1 ]   (233.08,75.86) -- (211.61,128.47) ;
\draw  [color={rgb, 255:red, 183; green, 181; blue, 181 }  ,draw opacity=1 ] (232.41,141.28) -- (176.27,173.45) -- (120.79,140.12) -- (121.45,74.61) -- (177.59,42.44) -- (233.07,75.78) -- cycle ;
\draw  [color={rgb, 255:red, 183; green, 181; blue, 181 }  ,draw opacity=1 ][dash pattern={on 3.75pt off 1.5pt on 3.75pt off 1.5pt}] (177.04,72.3) -- (233.08,75.86) -- (209.9,89.6) -- (146.94,89.2) -- (121.45,74.69) -- cycle ;
\draw [color={rgb, 255:red, 183; green, 181; blue, 181 }  ,draw opacity=1 ]   (121.45,74.69) -- (141.93,86.35) -- (146.94,89.2) ;
\draw [color={rgb, 255:red, 183; green, 181; blue, 181 }  ,draw opacity=1 ]   (209.9,89.6) -- (233.08,75.86) ;
\draw [color={rgb, 255:red, 183; green, 181; blue, 181 }  ,draw opacity=1 ]   (146.94,89.2) -- (209.9,89.6) ;
\draw [color={rgb, 255:red, 183; green, 181; blue, 181 }  ,draw opacity=1 ]   (120.79,140.12) -- (176.96,145.12) ;
\draw [color={rgb, 255:red, 183; green, 181; blue, 181 }  ,draw opacity=1 ]   (176.96,145.12) -- (232.42,141.36) ;
\draw [color={rgb, 255:red, 183; green, 181; blue, 181 }  ,draw opacity=1 ]   (146.94,89.2) -- (177.56,146.17) ;
\draw [color={rgb, 255:red, 183; green, 181; blue, 181 }  ,draw opacity=1 ]   (209.9,89.6) -- (176.37,146.17) ;
\draw [color={rgb, 255:red, 183; green, 181; blue, 181 }  ,draw opacity=1 ]   (209.9,89.6) -- (232.42,141.36) ;
\draw [color={rgb, 255:red, 183; green, 181; blue, 181 }  ,draw opacity=1 ]   (176.96,145.12) -- (176.27,173.53) ;
\draw [color={rgb, 255:red, 183; green, 181; blue, 181 }  ,draw opacity=1 ]   (177.6,42.53) -- (209.9,89.6) ;
\draw [color={rgb, 255:red, 183; green, 181; blue, 181 }  ,draw opacity=1 ]   (177.6,42.53) -- (146.94,89.2) ;
\draw [color={rgb, 255:red, 183; green, 181; blue, 181 }  ,draw opacity=1 ] [dash pattern={on 3.75pt off 1.5pt on 3.75pt off 1.5pt}]  (177.6,42.53) -- (177.04,72.3) ;
\draw  [color={rgb, 255:red, 183; green, 181; blue, 181 }  ,draw opacity=1 ][fill={rgb, 255:red, 0; green, 0; blue, 0 }  ,fill opacity=1 ] (174.38,145.12) .. controls (174.38,143.67) and (175.53,142.5) .. (176.96,142.5) .. controls (178.39,142.5) and (179.54,143.67) .. (179.54,145.12) .. controls (179.54,146.57) and (178.39,147.75) .. (176.96,147.75) .. controls (175.53,147.75) and (174.38,146.57) .. (174.38,145.12) -- cycle ;
\draw  [color={rgb, 255:red, 0; green, 0; blue, 0 }  ,draw opacity=1 ][fill={rgb, 255:red, 0; green, 0; blue, 0 }  ,fill opacity=1 ] (173.69,173.53) .. controls (173.69,172.08) and (174.85,170.91) .. (176.27,170.91) .. controls (177.7,170.91) and (178.86,172.08) .. (178.86,173.53) .. controls (178.86,174.98) and (177.7,176.16) .. (176.27,176.16) .. controls (174.85,176.16) and (173.69,174.98) .. (173.69,173.53) -- cycle ;
\draw  [color={rgb, 255:red, 183; green, 181; blue, 181 }  ,draw opacity=1 ][fill={rgb, 255:red, 0; green, 0; blue, 0 }  ,fill opacity=1 ] (229.84,141.36) .. controls (229.84,139.91) and (230.99,138.74) .. (232.42,138.74) .. controls (233.85,138.74) and (235,139.91) .. (235,141.36) .. controls (235,142.81) and (233.85,143.99) .. (232.42,143.99) .. controls (230.99,143.99) and (229.84,142.81) .. (229.84,141.36) -- cycle ;
\draw  [color={rgb, 255:red, 183; green, 181; blue, 181 }  ,draw opacity=1 ][fill={rgb, 255:red, 0; green, 0; blue, 0 }  ,fill opacity=1 ] (230.5,75.86) .. controls (230.5,74.41) and (231.66,73.24) .. (233.08,73.24) .. controls (234.51,73.24) and (235.67,74.41) .. (235.67,75.86) .. controls (235.67,77.31) and (234.51,78.49) .. (233.08,78.49) .. controls (231.66,78.49) and (230.5,77.31) .. (230.5,75.86) -- cycle ;
\draw  [color={rgb, 255:red, 183; green, 181; blue, 181 }  ,draw opacity=1 ][fill={rgb, 255:red, 0; green, 0; blue, 0 }  ,fill opacity=1 ] (174.45,72.3) .. controls (174.45,70.85) and (175.61,69.67) .. (177.04,69.67) .. controls (178.46,69.67) and (179.62,70.85) .. (179.62,72.3) .. controls (179.62,73.75) and (178.46,74.93) .. (177.04,74.93) .. controls (175.61,74.93) and (174.45,73.75) .. (174.45,72.3) -- cycle ;
\draw  [color={rgb, 255:red, 0; green, 0; blue, 0 }  ,draw opacity=1 ][fill={rgb, 255:red, 0; green, 0; blue, 0 }  ,fill opacity=1 ] (209.03,128.47) .. controls (209.03,127.02) and (210.19,125.84) .. (211.61,125.84) .. controls (213.04,125.84) and (214.2,127.02) .. (214.2,128.47) .. controls (214.2,129.92) and (213.04,131.09) .. (211.61,131.09) .. controls (210.19,131.09) and (209.03,129.92) .. (209.03,128.47) -- cycle ;
\draw [color={rgb, 255:red, 183; green, 181; blue, 181 }  ,draw opacity=1 ]   (146.94,89.2) -- (120.79,140.2) ;
\draw  [color={rgb, 255:red, 183; green, 181; blue, 181 }  ,draw opacity=1 ][fill={rgb, 255:red, 0; green, 0; blue, 0 }  ,fill opacity=1 ] (207.32,89.6) .. controls (207.32,88.15) and (208.48,86.98) .. (209.9,86.98) .. controls (211.33,86.98) and (212.49,88.15) .. (212.49,89.6) .. controls (212.49,91.05) and (211.33,92.23) .. (209.9,92.23) .. controls (208.48,92.23) and (207.32,91.05) .. (207.32,89.6) -- cycle ;
\draw [color={rgb, 255:red, 144; green, 19; blue, 254 }  ,draw opacity=1 ][line width=1.5]    (177.6,42.53) -- (146.94,89.2) ;
\draw [color={rgb, 255:red, 245; green, 166; blue, 35 }  ,draw opacity=1 ][line width=1.5]    (177.6,42.53) -- (121.45,74.69) ;
\draw [color={rgb, 255:red, 245; green, 166; blue, 35 }  ,draw opacity=1 ][line width=1.5]    (122.92,74.66) -- (146.94,89.2) ;
\draw  [color={rgb, 255:red, 0; green, 0; blue, 0 }  ,draw opacity=1 ][fill={rgb, 255:red, 0; green, 0; blue, 0 }  ,fill opacity=1 ] (118.87,74.69) .. controls (118.87,73.24) and (120.03,72.07) .. (121.45,72.07) .. controls (122.88,72.07) and (124.04,73.24) .. (124.04,74.69) .. controls (124.04,76.14) and (122.88,77.32) .. (121.45,77.32) .. controls (120.03,77.32) and (118.87,76.14) .. (118.87,74.69) -- cycle ;
\draw [color={rgb, 255:red, 65; green, 117; blue, 5 }  ,draw opacity=1 ][line width=1.5]    (177.6,42.53) -- (209.9,89.6) ;
\draw [color={rgb, 255:red, 65; green, 117; blue, 5 }  ,draw opacity=1 ][line width=1.5]    (146.94,89.2) -- (209.9,89.6) ;
\draw [color={rgb, 255:red, 74; green, 144; blue, 226 }  ,draw opacity=1 ][line width=1.5]    (177.6,42.52) -- (233.08,75.86) ;
\draw [color={rgb, 255:red, 74; green, 144; blue, 226 }  ,draw opacity=1 ][line width=1.5]    (176.96,145.12) -- (232.42,141.36) ;
\draw [color={rgb, 255:red, 74; green, 144; blue, 226 }  ,draw opacity=1 ][line width=1.5]    (233.08,75.86) -- (232.42,141.36) ;
\draw [color={rgb, 255:red, 74; green, 144; blue, 226 }  ,draw opacity=1 ][line width=1.5]    (146.94,89.2) -- (177.56,144.97) ;
\draw [color={rgb, 255:red, 126; green, 211; blue, 33 }  ,draw opacity=1 ][line width=1.5]  [dash pattern={on 3.75pt off 1.5pt on 3.75pt off 1.5pt}]  (177.6,42.53) -- (177.04,72.3) ;
\draw [color={rgb, 255:red, 126; green, 211; blue, 33 }  ,draw opacity=1 ][line width=1.5]    (120.79,140.12) -- (146.94,89.2) ;
\draw  [color={rgb, 255:red, 255; green, 0; blue, 0 }  ,draw opacity=1 ][fill={rgb, 255:red, 255; green, 0; blue, 0 }  ,fill opacity=1 ] (144.36,89.2) .. controls (144.36,87.75) and (145.52,86.58) .. (146.94,86.58) .. controls (148.37,86.58) and (149.53,87.75) .. (149.53,89.2) .. controls (149.53,90.65) and (148.37,91.83) .. (146.94,91.83) .. controls (145.52,91.83) and (144.36,90.65) .. (144.36,89.2) -- cycle ;
\draw  [color={rgb, 255:red, 0; green, 0; blue, 0 }  ,draw opacity=1 ][fill={rgb, 255:red, 0; green, 0; blue, 0 }  ,fill opacity=1 ] (142.51,128.06) .. controls (142.51,126.61) and (143.67,125.44) .. (145.09,125.44) .. controls (146.52,125.44) and (147.68,126.61) .. (147.68,128.06) .. controls (147.68,129.51) and (146.52,130.69) .. (145.09,130.69) .. controls (143.67,130.69) and (142.51,129.51) .. (142.51,128.06) -- cycle ;
\draw  [color={rgb, 255:red, 0; green, 0; blue, 0 }  ,draw opacity=1 ][fill={rgb, 255:red, 0; green, 0; blue, 0 }  ,fill opacity=1 ] (118.21,140.12) .. controls (118.21,138.67) and (119.36,137.49) .. (120.79,137.49) .. controls (122.22,137.49) and (123.37,138.67) .. (123.37,140.12) .. controls (123.37,141.57) and (122.22,142.74) .. (120.79,142.74) .. controls (119.36,142.74) and (118.21,141.57) .. (118.21,140.12) -- cycle ;
\draw  [color={rgb, 255:red, 0; green, 0; blue, 0 }  ,draw opacity=1 ][fill={rgb, 255:red, 0; green, 0; blue, 0 }  ,fill opacity=1 ] (174.38,145.12) .. controls (174.38,143.67) and (175.53,142.5) .. (176.96,142.5) .. controls (178.39,142.5) and (179.54,143.67) .. (179.54,145.12) .. controls (179.54,146.57) and (178.39,147.75) .. (176.96,147.75) .. controls (175.53,147.75) and (174.38,146.57) .. (174.38,145.12) -- cycle ;
\draw  [color={rgb, 255:red, 0; green, 0; blue, 0 }  ,draw opacity=1 ][fill={rgb, 255:red, 0; green, 0; blue, 0 }  ,fill opacity=1 ] (229.83,141.36) .. controls (229.83,139.91) and (230.99,138.74) .. (232.42,138.74) .. controls (233.84,138.74) and (235,139.91) .. (235,141.36) .. controls (235,142.81) and (233.84,143.99) .. (232.42,143.99) .. controls (230.99,143.99) and (229.83,142.81) .. (229.83,141.36) -- cycle ;
\draw  [color={rgb, 255:red, 0; green, 0; blue, 0 }  ,draw opacity=1 ][fill={rgb, 255:red, 0; green, 0; blue, 0 }  ,fill opacity=1 ] (230.49,75.78) .. controls (230.49,74.33) and (231.64,73.15) .. (233.07,73.15) .. controls (234.5,73.15) and (235.65,74.33) .. (235.65,75.78) .. controls (235.65,77.23) and (234.5,78.4) .. (233.07,78.4) .. controls (231.64,78.4) and (230.49,77.23) .. (230.49,75.78) -- cycle ;
\draw  [color={rgb, 255:red, 0; green, 0; blue, 0 }  ,draw opacity=1 ][fill={rgb, 255:red, 0; green, 0; blue, 0 }  ,fill opacity=1 ] (207.32,89.6) .. controls (207.32,88.15) and (208.48,86.98) .. (209.9,86.98) .. controls (211.33,86.98) and (212.49,88.15) .. (212.49,89.6) .. controls (212.49,91.05) and (211.33,92.23) .. (209.9,92.23) .. controls (208.48,92.23) and (207.32,91.05) .. (207.32,89.6) -- cycle ;
\draw  [color={rgb, 255:red, 255; green, 0; blue, 0 }  ,draw opacity=1 ][fill={rgb, 255:red, 255; green, 0; blue, 0 }  ,fill opacity=1 ] (175.02,42.53) .. controls (175.02,41.08) and (176.17,39.9) .. (177.6,39.9) .. controls (179.03,39.9) and (180.18,41.08) .. (180.18,42.53) .. controls (180.18,43.98) and (179.03,45.15) .. (177.6,45.15) .. controls (176.17,45.15) and (175.02,43.98) .. (175.02,42.53) -- cycle ;
\draw  [color={rgb, 255:red, 0; green, 0; blue, 0 }  ,draw opacity=1 ][fill={rgb, 255:red, 0; green, 0; blue, 0 }  ,fill opacity=1 ] (174.45,72.3) .. controls (174.45,70.85) and (175.61,69.67) .. (177.04,69.67) .. controls (178.46,69.67) and (179.62,70.85) .. (179.62,72.3) .. controls (179.62,73.75) and (178.46,74.93) .. (177.04,74.93) .. controls (175.61,74.93) and (174.45,73.75) .. (174.45,72.3) -- cycle ;
\draw [color={rgb, 255:red, 126; green, 211; blue, 33 }  ,draw opacity=1 ][line width=1.5]  [dash pattern={on 3.75pt off 1.5pt on 3.75pt off 1.5pt}]  (305.76,129.73) -- (337.7,73.97) ;
\draw [color={rgb, 255:red, 183; green, 181; blue, 181 }  ,draw opacity=1 ] [dash pattern={on 3.75pt off 1.5pt on 3.75pt off 1.5pt}]  (337.7,73.97) -- (372.28,130.13) ;
\draw [color={rgb, 255:red, 183; green, 181; blue, 181 }  ,draw opacity=1 ] [dash pattern={on 3.75pt off 1.5pt on 3.75pt off 1.5pt}]  (282.12,76.36) -- (305.76,129.73) ;
\draw [color={rgb, 255:red, 183; green, 181; blue, 181 }  ,draw opacity=1 ]   (393.75,77.53) -- (372.28,130.13) ;
\draw  [color={rgb, 255:red, 183; green, 181; blue, 181 }  ,draw opacity=1 ] (393.08,142.95) -- (336.94,175.12) -- (281.46,141.78) -- (282.12,76.28) -- (338.26,44.11) -- (393.74,77.44) -- cycle ;
\draw  [color={rgb, 255:red, 183; green, 181; blue, 181 }  ,draw opacity=1 ][dash pattern={on 3.75pt off 1.5pt on 3.75pt off 1.5pt}] (337.7,73.97) -- (393.75,77.53) -- (370.57,91.27) -- (307.61,90.87) -- (282.12,76.36) -- cycle ;
\draw [color={rgb, 255:red, 183; green, 181; blue, 181 }  ,draw opacity=1 ]   (282.12,76.36) -- (302.6,88.02) -- (307.61,90.87) ;
\draw [color={rgb, 255:red, 183; green, 181; blue, 181 }  ,draw opacity=1 ]   (370.57,91.27) -- (393.75,77.53) ;
\draw [color={rgb, 255:red, 183; green, 181; blue, 181 }  ,draw opacity=1 ]   (307.61,90.87) -- (370.57,91.27) ;
\draw [color={rgb, 255:red, 183; green, 181; blue, 181 }  ,draw opacity=1 ]   (307.61,90.87) -- (338.23,147.84) ;
\draw [color={rgb, 255:red, 183; green, 181; blue, 181 }  ,draw opacity=1 ]   (370.57,91.27) -- (337.04,147.84) ;
\draw [color={rgb, 255:red, 183; green, 181; blue, 181 }  ,draw opacity=1 ]   (370.57,91.27) -- (393.09,143.03) ;
\draw [color={rgb, 255:red, 183; green, 181; blue, 181 }  ,draw opacity=1 ]   (338.27,44.19) -- (370.57,91.27) ;
\draw [color={rgb, 255:red, 183; green, 181; blue, 181 }  ,draw opacity=1 ]   (338.27,44.19) -- (307.61,90.87) ;
\draw [color={rgb, 255:red, 183; green, 181; blue, 181 }  ,draw opacity=1 ] [dash pattern={on 3.75pt off 1.5pt on 3.75pt off 1.5pt}]  (338.27,44.19) -- (337.7,73.97) ;
\draw  [color={rgb, 255:red, 183; green, 181; blue, 181 }  ,draw opacity=1 ][fill={rgb, 255:red, 0; green, 0; blue, 0 }  ,fill opacity=1 ] (335.04,146.79) .. controls (335.04,145.34) and (336.2,144.16) .. (337.63,144.16) .. controls (339.05,144.16) and (340.21,145.34) .. (340.21,146.79) .. controls (340.21,148.24) and (339.05,149.41) .. (337.63,149.41) .. controls (336.2,149.41) and (335.04,148.24) .. (335.04,146.79) -- cycle ;
\draw  [color={rgb, 255:red, 0; green, 0; blue, 0 }  ,draw opacity=1 ][fill={rgb, 255:red, 0; green, 0; blue, 0 }  ,fill opacity=1 ] (334.36,175.2) .. controls (334.36,173.75) and (335.51,172.57) .. (336.94,172.57) .. controls (338.37,172.57) and (339.52,173.75) .. (339.52,175.2) .. controls (339.52,176.65) and (338.37,177.82) .. (336.94,177.82) .. controls (335.51,177.82) and (334.36,176.65) .. (334.36,175.2) -- cycle ;
\draw  [color={rgb, 255:red, 183; green, 181; blue, 181 }  ,draw opacity=1 ][fill={rgb, 255:red, 0; green, 0; blue, 0 }  ,fill opacity=1 ] (390.5,143.03) .. controls (390.5,141.58) and (391.66,140.41) .. (393.09,140.41) .. controls (394.51,140.41) and (395.67,141.58) .. (395.67,143.03) .. controls (395.67,144.48) and (394.51,145.66) .. (393.09,145.66) .. controls (391.66,145.66) and (390.5,144.48) .. (390.5,143.03) -- cycle ;
\draw  [color={rgb, 255:red, 183; green, 181; blue, 181 }  ,draw opacity=1 ][fill={rgb, 255:red, 0; green, 0; blue, 0 }  ,fill opacity=1 ] (391.17,77.53) .. controls (391.17,76.08) and (392.32,74.9) .. (393.75,74.9) .. controls (395.18,74.9) and (396.33,76.08) .. (396.33,77.53) .. controls (396.33,78.98) and (395.18,80.15) .. (393.75,80.15) .. controls (392.32,80.15) and (391.17,78.98) .. (391.17,77.53) -- cycle ;
\draw  [color={rgb, 255:red, 183; green, 181; blue, 181 }  ,draw opacity=1 ][fill={rgb, 255:red, 0; green, 0; blue, 0 }  ,fill opacity=1 ] (335.12,73.97) .. controls (335.12,72.52) and (336.28,71.34) .. (337.7,71.34) .. controls (339.13,71.34) and (340.29,72.52) .. (340.29,73.97) .. controls (340.29,75.42) and (339.13,76.59) .. (337.7,76.59) .. controls (336.28,76.59) and (335.12,75.42) .. (335.12,73.97) -- cycle ;
\draw  [color={rgb, 255:red, 0; green, 0; blue, 0 }  ,draw opacity=1 ][fill={rgb, 255:red, 0; green, 0; blue, 0 }  ,fill opacity=1 ] (369.7,130.13) .. controls (369.7,128.68) and (370.85,127.51) .. (372.28,127.51) .. controls (373.71,127.51) and (374.86,128.68) .. (374.86,130.13) .. controls (374.86,131.58) and (373.71,132.76) .. (372.28,132.76) .. controls (370.85,132.76) and (369.7,131.58) .. (369.7,130.13) -- cycle ;
\draw [color={rgb, 255:red, 183; green, 181; blue, 181 }  ,draw opacity=1 ]   (307.61,90.87) -- (281.46,141.86) ;
\draw  [color={rgb, 255:red, 183; green, 181; blue, 181 }  ,draw opacity=1 ][fill={rgb, 255:red, 0; green, 0; blue, 0 }  ,fill opacity=1 ] (367.99,91.27) .. controls (367.99,89.82) and (369.14,88.65) .. (370.57,88.65) .. controls (372,88.65) and (373.15,89.82) .. (373.15,91.27) .. controls (373.15,92.72) and (372,93.9) .. (370.57,93.9) .. controls (369.14,93.9) and (367.99,92.72) .. (367.99,91.27) -- cycle ;
\draw [color={rgb, 255:red, 144; green, 19; blue, 254 }  ,draw opacity=1 ][line width=1.5]    (338.27,44.19) -- (307.61,90.87) ;
\draw [color={rgb, 255:red, 245; green, 166; blue, 35 }  ,draw opacity=1 ][line width=1.5]    (338.27,44.19) -- (282.12,76.36) ;
\draw [color={rgb, 255:red, 245; green, 166; blue, 35 }  ,draw opacity=1 ][line width=1.5]    (282.12,76.28) -- (281.57,142.11) ;
\draw  [color={rgb, 255:red, 0; green, 0; blue, 0 }  ,draw opacity=1 ][fill={rgb, 255:red, 0; green, 0; blue, 0 }  ,fill opacity=1 ] (279.54,76.36) .. controls (279.54,74.91) and (280.69,73.73) .. (282.12,73.73) .. controls (283.55,73.73) and (284.7,74.91) .. (284.7,76.36) .. controls (284.7,77.81) and (283.55,78.98) .. (282.12,78.98) .. controls (280.69,78.98) and (279.54,77.81) .. (279.54,76.36) -- cycle ;
\draw [color={rgb, 255:red, 65; green, 117; blue, 5 }  ,draw opacity=1 ][line width=1.5]    (338.27,44.19) -- (370.57,91.27) ;
\draw [color={rgb, 255:red, 65; green, 117; blue, 5 }  ,draw opacity=1 ][line width=1.5]    (337.23,146.84) -- (370.57,91.27) ;
\draw [color={rgb, 255:red, 74; green, 144; blue, 226 }  ,draw opacity=1 ][line width=1.5]    (338.27,44.19) -- (393.74,77.53) ;
\draw [color={rgb, 255:red, 74; green, 144; blue, 226 }  ,draw opacity=1 ][line width=1.5]    (393.74,77.53) -- (393.09,143.03) ;
\draw [color={rgb, 255:red, 144; green, 19; blue, 254 }  ,draw opacity=1 ][line width=1.5]    (307.61,90.87) -- (337.63,146.79) ;
\draw [color={rgb, 255:red, 126; green, 211; blue, 33 }  ,draw opacity=1 ][line width=1.5]  [dash pattern={on 3.75pt off 1.5pt on 3.75pt off 1.5pt}]  (338.27,44.19) -- (337.7,73.97) ;
\draw  [color={rgb, 255:red, 255; green, 0; blue, 0 }  ,draw opacity=1 ][fill={rgb, 255:red, 255; green, 0; blue, 0 }  ,fill opacity=1 ] (335.04,146.79) .. controls (335.04,145.34) and (336.2,144.16) .. (337.63,144.16) .. controls (339.05,144.16) and (340.21,145.34) .. (340.21,146.79) .. controls (340.21,148.24) and (339.05,149.41) .. (337.63,149.41) .. controls (336.2,149.41) and (335.04,148.24) .. (335.04,146.79) -- cycle ;
\draw  [color={rgb, 255:red, 0; green, 0; blue, 0 }  ,draw opacity=1 ][fill={rgb, 255:red, 0; green, 0; blue, 0 }  ,fill opacity=1 ] (303.18,129.73) .. controls (303.18,128.28) and (304.33,127.11) .. (305.76,127.11) .. controls (307.19,127.11) and (308.34,128.28) .. (308.34,129.73) .. controls (308.34,131.18) and (307.19,132.36) .. (305.76,132.36) .. controls (304.33,132.36) and (303.18,131.18) .. (303.18,129.73) -- cycle ;
\draw  [color={rgb, 255:red, 0; green, 0; blue, 0 }  ,draw opacity=1 ][fill={rgb, 255:red, 0; green, 0; blue, 0 }  ,fill opacity=1 ] (278.87,141.78) .. controls (278.87,140.33) and (280.03,139.16) .. (281.46,139.16) .. controls (282.88,139.16) and (284.04,140.33) .. (284.04,141.78) .. controls (284.04,143.23) and (282.88,144.41) .. (281.46,144.41) .. controls (280.03,144.41) and (278.87,143.23) .. (278.87,141.78) -- cycle ;
\draw  [color={rgb, 255:red, 0; green, 0; blue, 0 }  ,draw opacity=1 ][fill={rgb, 255:red, 0; green, 0; blue, 0 }  ,fill opacity=1 ] (305.03,91.49) .. controls (305.03,90.04) and (306.19,88.87) .. (307.61,88.87) .. controls (309.04,88.87) and (310.19,90.04) .. (310.19,91.49) .. controls (310.19,92.94) and (309.04,94.12) .. (307.61,94.12) .. controls (306.19,94.12) and (305.03,92.94) .. (305.03,91.49) -- cycle ;
\draw  [color={rgb, 255:red, 0; green, 0; blue, 0 }  ,draw opacity=1 ][fill={rgb, 255:red, 0; green, 0; blue, 0 }  ,fill opacity=1 ] (390.5,143.03) .. controls (390.5,141.58) and (391.66,140.4) .. (393.08,140.4) .. controls (394.51,140.4) and (395.66,141.58) .. (395.66,143.03) .. controls (395.66,144.48) and (394.51,145.66) .. (393.08,145.66) .. controls (391.66,145.66) and (390.5,144.48) .. (390.5,143.03) -- cycle ;
\draw  [color={rgb, 255:red, 0; green, 0; blue, 0 }  ,draw opacity=1 ][fill={rgb, 255:red, 0; green, 0; blue, 0 }  ,fill opacity=1 ] (391.15,77.44) .. controls (391.15,75.99) and (392.31,74.82) .. (393.74,74.82) .. controls (395.16,74.82) and (396.32,75.99) .. (396.32,77.44) .. controls (396.32,78.89) and (395.16,80.07) .. (393.74,80.07) .. controls (392.31,80.07) and (391.15,78.89) .. (391.15,77.44) -- cycle ;
\draw  [color={rgb, 255:red, 0; green, 0; blue, 0 }  ,draw opacity=1 ][fill={rgb, 255:red, 0; green, 0; blue, 0 }  ,fill opacity=1 ] (367.99,91.27) .. controls (367.99,89.82) and (369.14,88.65) .. (370.57,88.65) .. controls (372,88.65) and (373.15,89.82) .. (373.15,91.27) .. controls (373.15,92.72) and (372,93.9) .. (370.57,93.9) .. controls (369.14,93.9) and (367.99,92.72) .. (367.99,91.27) -- cycle ;
\draw  [color={rgb, 255:red, 255; green, 0; blue, 0 }  ,draw opacity=1 ][fill={rgb, 255:red, 255; green, 0; blue, 0 }  ,fill opacity=1 ] (335.68,44.19) .. controls (335.68,42.74) and (336.84,41.57) .. (338.27,41.57) .. controls (339.69,41.57) and (340.85,42.74) .. (340.85,44.19) .. controls (340.85,45.64) and (339.69,46.82) .. (338.27,46.82) .. controls (336.84,46.82) and (335.68,45.64) .. (335.68,44.19) -- cycle ;
\draw  [color={rgb, 255:red, 0; green, 0; blue, 0 }  ,draw opacity=1 ][fill={rgb, 255:red, 0; green, 0; blue, 0 }  ,fill opacity=1 ] (335.12,73.97) .. controls (335.12,72.52) and (336.28,71.34) .. (337.7,71.34) .. controls (339.13,71.34) and (340.29,72.52) .. (340.29,73.97) .. controls (340.29,75.42) and (339.13,76.59) .. (337.7,76.59) .. controls (336.28,76.59) and (335.12,75.42) .. (335.12,73.97) -- cycle ;
\draw [color={rgb, 255:red, 126; green, 211; blue, 33 }  ,draw opacity=1 ][line width=1.5]  [dash pattern={on 3.75pt off 1.5pt on 3.75pt off 1.5pt}]  (462.09,129.06) -- (494.04,73.3) ;
\draw [color={rgb, 255:red, 183; green, 181; blue, 181 }  ,draw opacity=1 ] [dash pattern={on 3.75pt off 1.5pt on 3.75pt off 1.5pt}]  (494.04,73.3) -- (528.61,129.47) ;
\draw [color={rgb, 255:red, 183; green, 181; blue, 181 }  ,draw opacity=1 ] [dash pattern={on 3.75pt off 1.5pt on 3.75pt off 1.5pt}]  (438.45,75.69) -- (462.09,129.06) ;
\draw  [color={rgb, 255:red, 183; green, 181; blue, 181 }  ,draw opacity=1 ][dash pattern={on 3.75pt off 1.5pt on 3.75pt off 1.5pt}] (494.56,145.97) -- (549.42,142.36) -- (528.61,129.47) -- (462.09,129.06) -- (437.79,141.12) -- cycle ;
\draw [color={rgb, 255:red, 183; green, 181; blue, 181 }  ,draw opacity=1 ] [dash pattern={on 3.75pt off 1.5pt on 3.75pt off 1.5pt}]  (462.09,129.06) -- (493.27,174.53) ;
\draw  [color={rgb, 255:red, 183; green, 181; blue, 181 }  ,draw opacity=1 ][dash pattern={on 3.75pt off 1.5pt on 3.75pt off 1.5pt}] (494.04,73.3) -- (550.08,76.86) -- (526.9,90.6) -- (463.94,90.2) -- (438.45,75.69) -- cycle ;
\draw [color={rgb, 255:red, 183; green, 181; blue, 181 }  ,draw opacity=1 ]   (438.45,75.69) -- (458.93,87.35) -- (463.94,90.2) ;
\draw [color={rgb, 255:red, 183; green, 181; blue, 181 }  ,draw opacity=1 ]   (526.9,90.6) -- (550.08,76.86) ;
\draw [color={rgb, 255:red, 183; green, 181; blue, 181 }  ,draw opacity=1 ]   (463.94,90.2) -- (526.9,90.6) ;
\draw [color={rgb, 255:red, 183; green, 181; blue, 181 }  ,draw opacity=1 ]   (437.79,141.12) -- (493.96,146.12) ;
\draw [color={rgb, 255:red, 183; green, 181; blue, 181 }  ,draw opacity=1 ]   (493.96,146.12) -- (549.42,142.36) ;
\draw [color={rgb, 255:red, 183; green, 181; blue, 181 }  ,draw opacity=1 ]   (463.94,90.2) -- (494.56,147.17) ;
\draw [color={rgb, 255:red, 183; green, 181; blue, 181 }  ,draw opacity=1 ]   (526.9,90.6) -- (493.37,147.17) ;
\draw [color={rgb, 255:red, 183; green, 181; blue, 181 }  ,draw opacity=1 ]   (494.6,43.53) -- (526.9,90.6) ;
\draw [color={rgb, 255:red, 183; green, 181; blue, 181 }  ,draw opacity=1 ]   (494.6,43.53) -- (463.94,90.2) ;
\draw [color={rgb, 255:red, 183; green, 181; blue, 181 }  ,draw opacity=1 ] [dash pattern={on 3.75pt off 1.5pt on 3.75pt off 1.5pt}]  (494.6,43.53) -- (494.04,73.3) ;
\draw  [color={rgb, 255:red, 183; green, 181; blue, 181 }  ,draw opacity=1 ][fill={rgb, 255:red, 0; green, 0; blue, 0 }  ,fill opacity=1 ] (546.84,142.36) .. controls (546.84,140.91) and (547.99,139.74) .. (549.42,139.74) .. controls (550.85,139.74) and (552,140.91) .. (552,142.36) .. controls (552,143.81) and (550.85,144.99) .. (549.42,144.99) .. controls (547.99,144.99) and (546.84,143.81) .. (546.84,142.36) -- cycle ;
\draw  [color={rgb, 255:red, 183; green, 181; blue, 181 }  ,draw opacity=1 ][fill={rgb, 255:red, 0; green, 0; blue, 0 }  ,fill opacity=1 ] (547.5,76.86) .. controls (547.5,75.41) and (548.66,74.24) .. (550.08,74.24) .. controls (551.51,74.24) and (552.67,75.41) .. (552.67,76.86) .. controls (552.67,78.31) and (551.51,79.49) .. (550.08,79.49) .. controls (548.66,79.49) and (547.5,78.31) .. (547.5,76.86) -- cycle ;
\draw  [color={rgb, 255:red, 183; green, 181; blue, 181 }  ,draw opacity=1 ][fill={rgb, 255:red, 0; green, 0; blue, 0 }  ,fill opacity=1 ] (491.45,73.3) .. controls (491.45,71.85) and (492.61,70.67) .. (494.04,70.67) .. controls (495.46,70.67) and (496.62,71.85) .. (496.62,73.3) .. controls (496.62,74.75) and (495.46,75.93) .. (494.04,75.93) .. controls (492.61,75.93) and (491.45,74.75) .. (491.45,73.3) -- cycle ;
\draw [color={rgb, 255:red, 183; green, 181; blue, 181 }  ,draw opacity=1 ]   (463.94,90.2) -- (437.79,141.2) ;
\draw  [color={rgb, 255:red, 183; green, 181; blue, 181 }  ,draw opacity=1 ][fill={rgb, 255:red, 0; green, 0; blue, 0 }  ,fill opacity=1 ] (524.32,90.6) .. controls (524.32,89.15) and (525.48,87.98) .. (526.9,87.98) .. controls (528.33,87.98) and (529.49,89.15) .. (529.49,90.6) .. controls (529.49,92.05) and (528.33,93.23) .. (526.9,93.23) .. controls (525.48,93.23) and (524.32,92.05) .. (524.32,90.6) -- cycle ;
\draw [color={rgb, 255:red, 144; green, 19; blue, 254 }  ,draw opacity=1 ][line width=1.5]    (494.6,43.53) -- (463.94,90.2) ;
\draw [color={rgb, 255:red, 245; green, 166; blue, 35 }  ,draw opacity=1 ][line width=1.5]    (494.6,43.53) -- (438.45,75.69) ;
\draw [color={rgb, 255:red, 245; green, 166; blue, 35 }  ,draw opacity=1 ][line width=1.5]    (438.45,75.61) -- (437.91,141.44) ;
\draw  [color={rgb, 255:red, 0; green, 0; blue, 0 }  ,draw opacity=1 ][fill={rgb, 255:red, 0; green, 0; blue, 0 }  ,fill opacity=1 ] (435.87,75.69) .. controls (435.87,74.24) and (437.03,73.07) .. (438.45,73.07) .. controls (439.88,73.07) and (441.04,74.24) .. (441.04,75.69) .. controls (441.04,77.14) and (439.88,78.32) .. (438.45,78.32) .. controls (437.03,78.32) and (435.87,77.14) .. (435.87,75.69) -- cycle ;
\draw [color={rgb, 255:red, 65; green, 117; blue, 5 }  ,draw opacity=1 ][line width=1.5]    (494.6,43.53) -- (526.9,90.6) ;
\draw [color={rgb, 255:red, 65; green, 117; blue, 5 }  ,draw opacity=1 ][line width=1.5]    (549.42,142.36) -- (526.9,90.6) ;
\draw [color={rgb, 255:red, 74; green, 144; blue, 226 }  ,draw opacity=1 ][line width=1.5]    (494.6,43.52) -- (550.08,76.86) ;
\draw [color={rgb, 255:red, 74; green, 144; blue, 226 }  ,draw opacity=1 ][line width=1.5]  [dash pattern={on 3.75pt off 1.5pt on 3.75pt off 1.5pt}]  (493.27,174.45) -- (528.61,129.47) ;
\draw [color={rgb, 255:red, 144; green, 19; blue, 254 }  ,draw opacity=1 ][line width=1.5]    (463.94,90.2) -- (493.96,146.12) ;
\draw [color={rgb, 255:red, 126; green, 211; blue, 33 }  ,draw opacity=1 ][line width=1.5]  [dash pattern={on 3.75pt off 1.5pt on 3.75pt off 1.5pt}]  (494.6,43.53) -- (494.04,73.3) ;
\draw [color={rgb, 255:red, 126; green, 211; blue, 33 }  ,draw opacity=1 ][line width=1.5]  [dash pattern={on 3.75pt off 1.5pt on 3.75pt off 1.5pt}]  (462.1,129.15) -- (493.27,174.53) ;
\draw  [color={rgb, 255:red, 0; green, 0; blue, 0 }  ,draw opacity=1 ][fill={rgb, 255:red, 0; green, 0; blue, 0 }  ,fill opacity=1 ] (459.51,129.06) .. controls (459.51,127.61) and (460.67,126.44) .. (462.09,126.44) .. controls (463.52,126.44) and (464.68,127.61) .. (464.68,129.06) .. controls (464.68,130.51) and (463.52,131.69) .. (462.09,131.69) .. controls (460.67,131.69) and (459.51,130.51) .. (459.51,129.06) -- cycle ;
\draw  [color={rgb, 255:red, 0; green, 0; blue, 0 }  ,draw opacity=1 ][fill={rgb, 255:red, 0; green, 0; blue, 0 }  ,fill opacity=1 ] (435.21,141.12) .. controls (435.21,139.67) and (436.36,138.49) .. (437.79,138.49) .. controls (439.22,138.49) and (440.37,139.67) .. (440.37,141.12) .. controls (440.37,142.57) and (439.22,143.74) .. (437.79,143.74) .. controls (436.36,143.74) and (435.21,142.57) .. (435.21,141.12) -- cycle ;
\draw  [color={rgb, 255:red, 0; green, 0; blue, 0 }  ,draw opacity=1 ][fill={rgb, 255:red, 0; green, 0; blue, 0 }  ,fill opacity=1 ] (461.36,90.83) .. controls (461.36,89.38) and (462.52,88.2) .. (463.94,88.2) .. controls (465.37,88.2) and (466.53,89.38) .. (466.53,90.83) .. controls (466.53,92.28) and (465.37,93.45) .. (463.94,93.45) .. controls (462.52,93.45) and (461.36,92.28) .. (461.36,90.83) -- cycle ;
\draw  [color={rgb, 255:red, 0; green, 0; blue, 0 }  ,draw opacity=1 ][fill={rgb, 255:red, 0; green, 0; blue, 0 }  ,fill opacity=1 ] (547.49,76.78) .. controls (547.49,75.33) and (548.64,74.15) .. (550.07,74.15) .. controls (551.5,74.15) and (552.65,75.33) .. (552.65,76.78) .. controls (552.65,78.23) and (551.5,79.4) .. (550.07,79.4) .. controls (548.64,79.4) and (547.49,78.23) .. (547.49,76.78) -- cycle ;
\draw  [color={rgb, 255:red, 0; green, 0; blue, 0 }  ,draw opacity=1 ][fill={rgb, 255:red, 0; green, 0; blue, 0 }  ,fill opacity=1 ] (524.32,90.6) .. controls (524.32,89.15) and (525.48,87.98) .. (526.9,87.98) .. controls (528.33,87.98) and (529.49,89.15) .. (529.49,90.6) .. controls (529.49,92.05) and (528.33,93.23) .. (526.9,93.23) .. controls (525.48,93.23) and (524.32,92.05) .. (524.32,90.6) -- cycle ;
\draw  [color={rgb, 255:red, 255; green, 0; blue, 0 }  ,draw opacity=1 ][fill={rgb, 255:red, 255; green, 0; blue, 0 }  ,fill opacity=1 ] (492.02,43.53) .. controls (492.02,42.08) and (493.17,40.9) .. (494.6,40.9) .. controls (496.03,40.9) and (497.18,42.08) .. (497.18,43.53) .. controls (497.18,44.98) and (496.03,46.15) .. (494.6,46.15) .. controls (493.17,46.15) and (492.02,44.98) .. (492.02,43.53) -- cycle ;
\draw  [color={rgb, 255:red, 0; green, 0; blue, 0 }  ,draw opacity=1 ][fill={rgb, 255:red, 0; green, 0; blue, 0 }  ,fill opacity=1 ] (491.45,73.3) .. controls (491.45,71.85) and (492.61,70.67) .. (494.04,70.67) .. controls (495.46,70.67) and (496.62,71.85) .. (496.62,73.3) .. controls (496.62,74.75) and (495.46,75.93) .. (494.04,75.93) .. controls (492.61,75.93) and (491.45,74.75) .. (491.45,73.3) -- cycle ;
\draw  [color={rgb, 255:red, 0; green, 0; blue, 0 }  ,draw opacity=1 ][fill={rgb, 255:red, 0; green, 0; blue, 0 }  ,fill opacity=1 ] (491.38,146.12) .. controls (491.38,144.67) and (492.53,143.5) .. (493.96,143.5) .. controls (495.39,143.5) and (496.54,144.67) .. (496.54,146.12) .. controls (496.54,147.57) and (495.39,148.75) .. (493.96,148.75) .. controls (492.53,148.75) and (491.38,147.57) .. (491.38,146.12) -- cycle ;
\draw [color={rgb, 255:red, 65; green, 117; blue, 5 }  ,draw opacity=1 ][line width=1.5]    (549.42,142.36) -- (493.24,173.98) ;
\draw  [color={rgb, 255:red, 255; green, 0; blue, 0 }  ,draw opacity=1 ][fill={rgb, 255:red, 255; green, 0; blue, 0 }  ,fill opacity=1 ] (490.66,173.98) .. controls (490.66,172.53) and (491.82,171.35) .. (493.24,171.35) .. controls (494.67,171.35) and (495.83,172.53) .. (495.83,173.98) .. controls (495.83,175.43) and (494.67,176.6) .. (493.24,176.6) .. controls (491.82,176.6) and (490.66,175.43) .. (490.66,173.98) -- cycle ;
\draw  [color={rgb, 255:red, 0; green, 0; blue, 0 }  ,draw opacity=1 ][fill={rgb, 255:red, 0; green, 0; blue, 0 }  ,fill opacity=1 ] (546.84,142.36) .. controls (546.84,140.91) and (547.99,139.74) .. (549.42,139.74) .. controls (550.85,139.74) and (552,140.91) .. (552,142.36) .. controls (552,143.81) and (550.85,144.99) .. (549.42,144.99) .. controls (547.99,144.99) and (546.84,143.81) .. (546.84,142.36) -- cycle ;
\draw  [color={rgb, 255:red, 0; green, 0; blue, 0 }  ,draw opacity=1 ][fill={rgb, 255:red, 0; green, 0; blue, 0 }  ,fill opacity=1 ] (526.03,129.47) .. controls (526.03,128.02) and (527.19,126.84) .. (528.61,126.84) .. controls (530.04,126.84) and (531.2,128.02) .. (531.2,129.47) .. controls (531.2,130.92) and (530.04,132.09) .. (528.61,132.09) .. controls (527.19,132.09) and (526.03,130.92) .. (526.03,129.47) -- cycle ;

\end{tikzpicture}

%% file: Figures/T_t.tikz
\tikzset{every picture/.style={line width=0.75pt}} 

\begin{tikzpicture}[x=0.75pt,y=0.75pt,yscale=-1,xscale=1]

\draw  [draw opacity=0][fill={rgb, 255:red, 74; green, 144; blue, 226 }  ,fill opacity=0.1 ] (531.83,200.17) -- (416.83,199.17) -- (444,131.08) -- (464.5,112.5) -- (485.83,106.17) -- (510,135.08) -- cycle ;
\draw  [color={rgb, 255:red, 65; green, 117; blue, 5 }  ,draw opacity=1 ] (416.83,199.17) -- (444,131.08) -- (464.5,112.5) -- cycle ;
\draw [color={rgb, 255:red, 65; green, 117; blue, 5 }  ,draw opacity=0.2 ] [dash pattern={on 1.5pt off 0.75pt}]  (485.83,106.17) -- (444,131.08) ;
\draw  [color={rgb, 255:red, 65; green, 117; blue, 5 }  ,draw opacity=1 ][line width=0.75]  (510,135.08) -- (531.83,200.17) -- (464.5,112.5) -- (485.83,106.17) -- cycle ;
\draw  [color={rgb, 255:red, 208; green, 2; blue, 27 }  ,draw opacity=1 ][fill={rgb, 255:red, 74; green, 144; blue, 226 }  ,fill opacity=0.1 ] (244.33,199.42) -- (129.33,198.42) -- (177,111.75) -- (198.33,105.42) -- cycle ;
\draw [color={rgb, 255:red, 208; green, 2; blue, 27 }  ,draw opacity=1 ]   (177,111.75) -- (244.33,199.42) ;
\draw [color={rgb, 255:red, 208; green, 2; blue, 27 }  ,draw opacity=1 ][fill={rgb, 255:red, 255; green, 0; blue, 0 }  ,fill opacity=0.25 ]   (464.5,112.5) -- (531.83,200.17) ;
\draw [color={rgb, 255:red, 65; green, 117; blue, 5 }  ,draw opacity=0.2 ] [dash pattern={on 1.5pt off 0.75pt}]  (480.56,171.72) -- (416.83,199.17) ;
\draw [color={rgb, 255:red, 65; green, 117; blue, 5 }  ,draw opacity=0.2 ] [dash pattern={on 1.5pt off 0.75pt}]  (480.56,171.72) -- (531.83,200.17) ;
\draw [color={rgb, 255:red, 65; green, 117; blue, 5 }  ,draw opacity=0.2 ] [dash pattern={on 1.5pt off 0.75pt}]  (485.83,106.17) -- (480.56,171.72) ;
\draw [color={rgb, 255:red, 65; green, 117; blue, 5 }  ,draw opacity=1 ]   (510,135.08) -- (464.5,112.5) ;
\draw  [color={rgb, 255:red, 208; green, 2; blue, 27 }  ,draw opacity=1 ][fill={rgb, 255:red, 208; green, 2; blue, 27 }  ,fill opacity=1 ] (196.05,105.42) .. controls (196.05,104.16) and (197.07,103.13) .. (198.33,103.13) .. controls (199.59,103.13) and (200.62,104.16) .. (200.62,105.42) .. controls (200.62,106.68) and (199.59,107.7) .. (198.33,107.7) .. controls (197.07,107.7) and (196.05,106.68) .. (196.05,105.42) -- cycle ;
\draw  [color={rgb, 255:red, 208; green, 2; blue, 27 }  ,draw opacity=1 ][fill={rgb, 255:red, 208; green, 2; blue, 27 }  ,fill opacity=1 ] (174.72,111.75) .. controls (174.72,110.49) and (175.74,109.47) .. (177,109.47) .. controls (178.26,109.47) and (179.28,110.49) .. (179.28,111.75) .. controls (179.28,113.01) and (178.26,114.03) .. (177,114.03) .. controls (175.74,114.03) and (174.72,113.01) .. (174.72,111.75) -- cycle ;
\draw  [color={rgb, 255:red, 208; green, 2; blue, 27 }  ,draw opacity=1 ][fill={rgb, 255:red, 208; green, 2; blue, 27 }  ,fill opacity=1 ] (242.05,199.42) .. controls (242.05,198.16) and (243.07,197.13) .. (244.33,197.13) .. controls (245.59,197.13) and (246.62,198.16) .. (246.62,199.42) .. controls (246.62,200.68) and (245.59,201.7) .. (244.33,201.7) .. controls (243.07,201.7) and (242.05,200.68) .. (242.05,199.42) -- cycle ;
\draw  [color={rgb, 255:red, 208; green, 2; blue, 27 }  ,draw opacity=1 ][fill={rgb, 255:red, 208; green, 2; blue, 27 }  ,fill opacity=1 ] (127.05,198.42) .. controls (127.05,197.16) and (128.07,196.13) .. (129.33,196.13) .. controls (130.59,196.13) and (131.62,197.16) .. (131.62,198.42) .. controls (131.62,199.68) and (130.59,200.7) .. (129.33,200.7) .. controls (128.07,200.7) and (127.05,199.68) .. (127.05,198.42) -- cycle ;
\draw  [color={rgb, 255:red, 208; green, 2; blue, 27 }  ,draw opacity=1 ][fill={rgb, 255:red, 208; green, 2; blue, 27 }  ,fill opacity=1 ] (483.55,106.17) .. controls (483.55,104.91) and (484.57,103.88) .. (485.83,103.88) .. controls (487.09,103.88) and (488.12,104.91) .. (488.12,106.17) .. controls (488.12,107.43) and (487.09,108.45) .. (485.83,108.45) .. controls (484.57,108.45) and (483.55,107.43) .. (483.55,106.17) -- cycle ;
\draw  [draw opacity=0][fill={rgb, 255:red, 189; green, 200; blue, 176 }  ,fill opacity=1 ] (478.08,171.72) .. controls (478.08,170.36) and (479.19,169.25) .. (480.56,169.25) .. controls (481.92,169.25) and (483.03,170.36) .. (483.03,171.72) .. controls (483.03,173.09) and (481.92,174.2) .. (480.56,174.2) .. controls (479.19,174.2) and (478.08,173.09) .. (478.08,171.72) -- cycle ;
\draw  [color={rgb, 255:red, 65; green, 117; blue, 5 }  ,draw opacity=1 ][fill={rgb, 255:red, 65; green, 117; blue, 5 }  ,fill opacity=1 ] (507.72,135.08) .. controls (507.72,133.82) and (508.74,132.8) .. (510,132.8) .. controls (511.26,132.8) and (512.28,133.82) .. (512.28,135.08) .. controls (512.28,136.34) and (511.26,137.37) .. (510,137.37) .. controls (508.74,137.37) and (507.72,136.34) .. (507.72,135.08) -- cycle ;
\draw  [color={rgb, 255:red, 65; green, 117; blue, 5 }  ,draw opacity=1 ][fill={rgb, 255:red, 65; green, 117; blue, 5 }  ,fill opacity=1 ] (441.72,131.08) .. controls (441.72,129.82) and (442.74,128.8) .. (444,128.8) .. controls (445.26,128.8) and (446.28,129.82) .. (446.28,131.08) .. controls (446.28,132.34) and (445.26,133.37) .. (444,133.37) .. controls (442.74,133.37) and (441.72,132.34) .. (441.72,131.08) -- cycle ;
\draw [color={rgb, 255:red, 65; green, 117; blue, 5 }  ,draw opacity=1 ]   (464.5,112.5) -- (461.89,170.39) ;
\draw [color={rgb, 255:red, 65; green, 117; blue, 5 }  ,draw opacity=1 ]   (461.89,170.39) -- (416.83,199.17) ;
\draw [color={rgb, 255:red, 65; green, 117; blue, 5 }  ,draw opacity=1 ]   (461.89,170.39) -- (531.83,200.17) ;
\draw  [color={rgb, 255:red, 65; green, 117; blue, 5 }  ,draw opacity=1 ][fill={rgb, 255:red, 65; green, 117; blue, 5 }  ,fill opacity=1 ] (459.61,170.39) .. controls (459.61,169.13) and (460.63,168.11) .. (461.89,168.11) .. controls (463.15,168.11) and (464.17,169.13) .. (464.17,170.39) .. controls (464.17,171.65) and (463.15,172.67) .. (461.89,172.67) .. controls (460.63,172.67) and (459.61,171.65) .. (459.61,170.39) -- cycle ;
\draw  [color={rgb, 255:red, 208; green, 2; blue, 27 }  ,draw opacity=1 ][fill={rgb, 255:red, 208; green, 2; blue, 27 }  ,fill opacity=1 ] (462.22,112.5) .. controls (462.22,111.24) and (463.24,110.22) .. (464.5,110.22) .. controls (465.76,110.22) and (466.78,111.24) .. (466.78,112.5) .. controls (466.78,113.76) and (465.76,114.78) .. (464.5,114.78) .. controls (463.24,114.78) and (462.22,113.76) .. (462.22,112.5) -- cycle ;
\draw  [color={rgb, 255:red, 208; green, 2; blue, 27 }  ,draw opacity=1 ][fill={rgb, 255:red, 208; green, 2; blue, 27 }  ,fill opacity=1 ] (414.55,199.17) .. controls (414.55,197.91) and (415.57,196.88) .. (416.83,196.88) .. controls (418.09,196.88) and (419.12,197.91) .. (419.12,199.17) .. controls (419.12,200.43) and (418.09,201.45) .. (416.83,201.45) .. controls (415.57,201.45) and (414.55,200.43) .. (414.55,199.17) -- cycle ;
\draw  [color={rgb, 255:red, 208; green, 2; blue, 27 }  ,draw opacity=1 ][fill={rgb, 255:red, 208; green, 2; blue, 27 }  ,fill opacity=1 ] (529.55,200.17) .. controls (529.55,198.91) and (530.57,197.88) .. (531.83,197.88) .. controls (533.09,197.88) and (534.12,198.91) .. (534.12,200.17) .. controls (534.12,201.43) and (533.09,202.45) .. (531.83,202.45) .. controls (530.57,202.45) and (529.55,201.43) .. (529.55,200.17) -- cycle ;
\draw [color={rgb, 255:red, 208; green, 2; blue, 27 }  ,draw opacity=0.2 ] [dash pattern={on 1.5pt off 0.75pt}]  (416.83,199.17) -- (485.83,106.17) ;
\draw [color={rgb, 255:red, 208; green, 2; blue, 27 }  ,draw opacity=0.2 ] [dash pattern={on 1.5pt off 0.75pt}]  (129.33,198.42) -- (198.33,105.42) ;
\draw [color={rgb, 255:red, 208; green, 2; blue, 27 }  ,draw opacity=1 ]   (464.5,112.5) -- (416.83,199.17) ;
\draw [color={rgb, 255:red, 208; green, 2; blue, 27 }  ,draw opacity=1 ]   (531.83,200.17) -- (416.83,199.17) ;
\draw [color={rgb, 255:red, 208; green, 2; blue, 27 }  ,draw opacity=1 ]   (485.83,106.17) -- (464.5,112.5) ;
\draw [color={rgb, 255:red, 208; green, 2; blue, 27 }  ,draw opacity=0.2 ] [dash pattern={on 1.5pt off 0.75pt}]  (531.83,200.17) -- (485.83,106.17) ;

\end{tikzpicture}

%% file: Figures/partial_order_cube.tikz
\tikzset{every picture/.style={line width=0.75pt}} 

\begin{tikzpicture}[x=0.75pt,y=0.75pt,yscale=-1,xscale=1]

\draw  [draw opacity=0][fill={rgb, 255:red, 155; green, 155; blue, 155 }  ,fill opacity=0.18 ] (130.62,162.15) -- (195.35,219.45) -- (130.62,280.38) -- (65.88,224) -- cycle ;
\draw [color={rgb, 255:red, 0; green, 0; blue, 0 }  ,draw opacity=1 ]   (130.62,162.15) -- (65.88,224) ;
\draw [color={rgb, 255:red, 0; green, 0; blue, 0 }  ,draw opacity=1 ]   (130.62,162.15) -- (195.35,219.45) ;
\draw    (147.01,233.09) -- (195.35,219.45) ;
\draw    (130.62,280.38) -- (195.35,219.45) ;
\draw    (130.62,280.38) -- (65.88,224) ;
\draw    (130.62,280.38) -- (147.01,233.09) ;
\draw  [dash pattern={on 2.25pt off 1.5pt}]  (130.62,280.38) -- (114.23,209.44) ;
\draw  [dash pattern={on 2.25pt off 1.5pt}]  (114.23,209.44) -- (65.88,224) ;
\draw  [dash pattern={on 2.25pt off 1.5pt}]  (195.35,219.45) -- (114.23,209.44) ;
\draw [color={rgb, 255:red, 0; green, 0; blue, 0 }  ,draw opacity=1 ] [dash pattern={on 2.25pt off 1.5pt}]  (130.62,162.15) -- (114.23,209.44) ;
\draw    (65.88,224) -- (147.01,233.09) ;
\draw [color={rgb, 255:red, 0; green, 0; blue, 0 }  ,draw opacity=1 ]   (253.83,227.92) -- (276.62,255.43) ;
\draw [color={rgb, 255:red, 0; green, 0; blue, 0 }  ,draw opacity=1 ]   (290.68,271.85) -- (293.44,275.18) -- (311.11,296.5) ;
\draw [color={rgb, 255:red, 0; green, 0; blue, 0 }  ,draw opacity=1 ]   (311.77,271.85) -- (313.34,276.84) -- (319.32,295.75) ;
\draw [color={rgb, 255:red, 0; green, 0; blue, 0 }  ,draw opacity=1 ]   (297.89,228.69) -- (306.66,256.08) ;
\draw [color={rgb, 255:red, 0; green, 0; blue, 0 }  ,draw opacity=1 ]   (343.52,227.79) -- (336.69,256.08) ;
\draw [color={rgb, 255:red, 0; green, 0; blue, 0 }  ,draw opacity=1 ]   (332.54,272.18) -- (331.31,277.02) -- (326.52,295.85) ;
\draw [color={rgb, 255:red, 0; green, 0; blue, 0 }  ,draw opacity=1 ]   (388.59,227.92) -- (366.09,256.4) ;
\draw [color={rgb, 255:red, 0; green, 0; blue, 0 }  ,draw opacity=1 ]   (352.67,272.5) -- (347.89,278.48) -- (333.48,296.5) ;
\draw [color={rgb, 255:red, 0; green, 0; blue, 0 }  ,draw opacity=1 ]   (253.38,206.47) -- (276.17,179.15) ;
\draw [color={rgb, 255:red, 0; green, 0; blue, 0 }  ,draw opacity=1 ]   (290.23,162.84) -- (292.99,159.53) -- (310.67,138.35) ;
\draw [color={rgb, 255:red, 0; green, 0; blue, 0 }  ,draw opacity=1 ]   (311.32,162.84) -- (312.9,157.88) -- (318.87,139.1) ;
\draw [color={rgb, 255:red, 0; green, 0; blue, 0 }  ,draw opacity=1 ]   (297.44,205.71) -- (306.21,178.51) ;
\draw [color={rgb, 255:red, 0; green, 0; blue, 0 }  ,draw opacity=1 ]   (343.07,206.61) -- (336.25,178.51) ;
\draw [color={rgb, 255:red, 0; green, 0; blue, 0 }  ,draw opacity=1 ]   (332.09,162.52) -- (330.86,157.7) -- (326.08,139) ;
\draw [color={rgb, 255:red, 0; green, 0; blue, 0 }  ,draw opacity=1 ]   (388.15,206.47) -- (365.65,178.19) ;
\draw [color={rgb, 255:red, 0; green, 0; blue, 0 }  ,draw opacity=1 ]   (352.22,162.2) -- (347.44,156.25) -- (333.03,138.35) ;
\draw [color={rgb, 255:red, 0; green, 0; blue, 0 }  ,draw opacity=1 ]   (455.11,163.08) -- (455.01,209.89) ;
\draw [color={rgb, 255:red, 0; green, 0; blue, 0 }  ,draw opacity=1 ]   (455.2,226.65) -- (455.01,271.49) ;
\draw [color={rgb, 255:red, 0; green, 0; blue, 0 }  ,draw opacity=1 ]   (596.12,225.24) -- (595.93,272.06) ;
\draw [color={rgb, 255:red, 0; green, 0; blue, 0 }  ,draw opacity=1 ]   (595.26,165.14) -- (595.93,210.46) ;
\draw [color={rgb, 255:red, 0; green, 0; blue, 0 }  ,draw opacity=1 ]   (465.08,271.49) .. controls (473.58,255.28) and (475.37,243.57) .. (474.92,228.24) ;
\draw [color={rgb, 255:red, 0; green, 0; blue, 0 }  ,draw opacity=1 ]   (520.55,271.96) .. controls (510.67,252.4) and (508.47,243.57) .. (508.03,229.14) ;
\draw [color={rgb, 255:red, 0; green, 0; blue, 0 }  ,draw opacity=1 ]   (587.66,270.61) .. controls (575.58,254.83) and (573.79,244.02) .. (573.34,228.24) ;
\draw [color={rgb, 255:red, 0; green, 0; blue, 0 }  ,draw opacity=1 ]   (530.84,272.41) .. controls (539.47,252.8) and (541.58,244.47) .. (542.25,229.59) ;
\draw [color={rgb, 255:red, 0; green, 0; blue, 0 }  ,draw opacity=1 ]   (463.74,164.24) .. controls (484.32,173.03) and (504.27,195.6) .. (507.58,208.86) ;
\draw [color={rgb, 255:red, 0; green, 0; blue, 0 }  ,draw opacity=1 ]   (519.32,165.13) .. controls (499.07,174.4) and (483.87,186.4) .. (474.92,209.31) ;
\draw [color={rgb, 255:red, 0; green, 0; blue, 0 }  ,draw opacity=1 ]   (526.43,164.96) .. controls (552.67,176) and (567.47,192.8) .. (571.62,208.91) ;
\draw [color={rgb, 255:red, 0; green, 0; blue, 0 }  ,draw opacity=1 ]   (586.31,164.69) .. controls (567.53,173.25) and (546.16,197.09) .. (543.07,209.6) ;
\draw [color={rgb, 255:red, 189; green, 16; blue, 224 }  ,draw opacity=1 ]   (131.28,222.38) -- (153.17,203.62) ;
\draw [shift={(155.44,201.67)}, rotate = 139.39] [fill={rgb, 255:red, 189; green, 16; blue, 224 }  ,fill opacity=1 ][line width=0.08]  [draw opacity=0] (10.72,-5.15) -- (0,0) -- (10.72,5.15) -- (7.12,0) -- cycle    ;
\draw [color={rgb, 255:red, 255; green, 0; blue, 30 }  ,draw opacity=1 ]   (131.28,222.38) -- (131.43,195.33) ;
\draw [shift={(131.44,192.33)}, rotate = 90.31] [fill={rgb, 255:red, 255; green, 0; blue, 30 }  ,fill opacity=1 ][line width=0.08]  [draw opacity=0] (10.72,-5.15) -- (0,0) -- (10.72,5.15) -- (7.12,0) -- cycle    ;
\draw [color={rgb, 255:red, 0; green, 0; blue, 0 }  ,draw opacity=1 ]   (130.62,162.15) -- (147.01,233.09) ;
\draw  [draw opacity=0][fill={rgb, 255:red, 155; green, 155; blue, 155 }  ,fill opacity=0.21 ] (130.62,162.15) -- (195.35,219.45) -- (130.62,280.38) -- (147.01,233.09) -- cycle ;

\draw (128.23,150.57) node [anchor=north west][inner sep=0.75pt]  [font=\scriptsize] [align=left] {f};
\draw (126.45,282.65) node [anchor=north west][inner sep=0.75pt]  [font=\scriptsize] [align=left] {a};
\draw (56.86,219.2) node [anchor=north west][inner sep=0.75pt]  [font=\scriptsize] [align=left] {e};
\draw (197.07,213.66) node [anchor=north west][inner sep=0.75pt]  [font=\scriptsize] [align=left] {c};
\draw (101.08,197.09) node [anchor=north west][inner sep=0.75pt]  [font=\scriptsize] [align=left] {d};
\draw (148.32,235.35) node [anchor=north west][inner sep=0.75pt]  [font=\scriptsize] [align=left] {b};
\draw (319.99,127.12) node [anchor=north west][inner sep=0.75pt]  [font=\scriptsize] [align=left] {f};
\draw (318.55,299.95) node [anchor=north west][inner sep=0.75pt]  [font=\scriptsize] [align=left] {a};
\draw (384.76,212.26) node [anchor=north west][inner sep=0.75pt]  [font=\scriptsize] [align=left] {e};
\draw (295.16,211.67) node [anchor=north west][inner sep=0.75pt]  [font=\scriptsize] [align=left] {c};
\draw (339.9,211.74) node [anchor=north west][inner sep=0.75pt]  [font=\scriptsize] [align=left] {d};
\draw (250.62,212.02) node [anchor=north west][inner sep=0.75pt]  [font=\scriptsize] [align=left] {b};
\draw (521.29,154.68) node [anchor=north west][inner sep=0.75pt]  [font=\scriptsize] [align=left] {f};
\draw (522.59,271.63) node [anchor=north west][inner sep=0.75pt]  [font=\scriptsize] [align=left] {a};
\draw (587.53,271.95) node [anchor=north west][inner sep=0.75pt]  [font=\scriptsize] [align=left] {e};
\draw (457.91,154.08) node [anchor=north west][inner sep=0.75pt]  [font=\scriptsize] [align=left] {c};
\draw (585.13,153.86) node [anchor=north west][inner sep=0.75pt]  [font=\scriptsize] [align=left] {b};
\draw (456.66,271.4) node [anchor=north west][inner sep=0.75pt]  [font=\scriptsize] [align=left] {d};
\draw (271.17,258.57) node [anchor=north west][inner sep=0.75pt]  [font=\scriptsize] [align=left] {\{a,b\}};
\draw (298.73,258.75) node [anchor=north west][inner sep=0.75pt]  [font=\scriptsize] [align=left] {\{a,c\}};
\draw (325.03,258.4) node [anchor=north west][inner sep=0.75pt]  [font=\scriptsize] [align=left] {\{a,d\}};
\draw (351.28,258.25) node [anchor=north west][inner sep=0.75pt]  [font=\scriptsize] [align=left] {\{a,e\}};
\draw (273.6,164.55) node [anchor=north west][inner sep=0.75pt]  [font=\scriptsize] [align=left] {\{b,f\}};
\draw (299.66,164.23) node [anchor=north west][inner sep=0.75pt]  [font=\scriptsize] [align=left] {\{c,f\}};
\draw (325.97,165.37) node [anchor=north west][inner sep=0.75pt]  [font=\scriptsize] [align=left] {\{d,f\}};
\draw (351.71,165.23) node [anchor=north west][inner sep=0.75pt]  [font=\scriptsize] [align=left] {\{e,f\}};
\draw (438.3,213.23) node [anchor=north west][inner sep=0.75pt]  [font=\scriptsize] [align=left] {\{c,d\}};
\draw (586.54,213.28) node [anchor=north west][inner sep=0.75pt]  [font=\scriptsize] [align=left] {\{b,e\}};
\draw (465.51,213.78) node [anchor=north west][inner sep=0.75pt]  [font=\scriptsize] [align=left] {\{d,f\}};
\draw (496.12,213.23) node [anchor=north west][inner sep=0.75pt]  [font=\scriptsize] [align=left] {\{a,c\}};
\draw (529.6,213.28) node [anchor=north west][inner sep=0.75pt]  [font=\scriptsize] [align=left] {\{a,b\}};
\draw (560.82,212.78) node [anchor=north west][inner sep=0.75pt]  [font=\scriptsize] [align=left] {\{e,f\}};
\draw (125.58,181.09) node [anchor=north west][inner sep=0.75pt]  [font=\scriptsize] [align=left] {\textcolor[rgb]{1,0,0}{u}};
\draw (155.08,195.09) node [anchor=north west][inner sep=0.75pt]  [font=\scriptsize] [align=left] {\textcolor[rgb]{0.74,0.06,0.88}{w}};

\end{tikzpicture}

%% file: Figures/grados_cubo1.tikz
\tikzset{every picture/.style={line width=0.75pt}} 

\begin{tikzpicture}[x=0.75pt,y=0.75pt,yscale=-1.2,xscale=1.2]

\draw [fill={rgb, 255:red, 0; green, 0; blue, 0 }  ,fill opacity=1 ]   (264.04,70.43) -- (263.78,130.39) ;
\draw [fill={rgb, 255:red, 0; green, 0; blue, 0 }  ,fill opacity=1 ]   (92.67,100.83) -- (153,70.17) ;
\draw [fill={rgb, 255:red, 0; green, 0; blue, 0 }  ,fill opacity=1 ]   (92.67,100.83) -- (153,130.17) ;
\draw  [color={rgb, 255:red, 0; green, 0; blue, 0 }  ,draw opacity=1 ][fill={rgb, 255:red, 0; green, 0; blue, 0 }  ,fill opacity=1 ][line width=2.25]  (91.06,100.83) .. controls (91.06,99.97) and (91.78,99.28) .. (92.67,99.28) .. controls (93.55,99.28) and (94.27,99.97) .. (94.27,100.83) .. controls (94.27,101.69) and (93.55,102.39) .. (92.67,102.39) .. controls (91.78,102.39) and (91.06,101.69) .. (91.06,100.83) -- cycle ;
\draw [fill={rgb, 255:red, 0; green, 0; blue, 0 }  ,fill opacity=1 ]   (92.67,100.83) -- (153.33,82.17) ;
\draw [fill={rgb, 255:red, 0; green, 0; blue, 0 }  ,fill opacity=1 ]   (92.67,100.83) -- (153,117.17) ;
\draw  [color={rgb, 255:red, 0; green, 0; blue, 0 }  ,draw opacity=1 ][fill={rgb, 255:red, 0; green, 0; blue, 0 }  ,fill opacity=1 ][line width=2.25]  (262.51,70.61) .. controls (262.51,69.75) and (263.23,69.06) .. (264.11,69.06) .. controls (265,69.06) and (265.71,69.75) .. (265.71,70.61) .. controls (265.71,71.47) and (265,72.17) .. (264.11,72.17) .. controls (263.23,72.17) and (262.51,71.47) .. (262.51,70.61) -- cycle ;
\draw [fill={rgb, 255:red, 0; green, 0; blue, 0 }  ,fill opacity=1 ]   (264.11,70.61) -- (323.44,52.94) ;
\draw [fill={rgb, 255:red, 0; green, 0; blue, 0 }  ,fill opacity=1 ]   (264.11,70.61) -- (323.44,86.94) ;
\draw [fill={rgb, 255:red, 0; green, 0; blue, 0 }  ,fill opacity=1 ]   (263.78,130.39) -- (323.11,113.72) ;
\draw [fill={rgb, 255:red, 0; green, 0; blue, 0 }  ,fill opacity=1 ]   (263.78,130.39) -- (323.11,146.72) ;
\draw  [color={rgb, 255:red, 0; green, 0; blue, 0 }  ,draw opacity=1 ][fill={rgb, 255:red, 0; green, 0; blue, 0 }  ,fill opacity=1 ][line width=2.25]  (262.18,130.39) .. controls (262.18,129.53) and (262.89,128.83) .. (263.78,128.83) .. controls (264.66,128.83) and (265.38,129.53) .. (265.38,130.39) .. controls (265.38,131.25) and (264.66,131.94) .. (263.78,131.94) .. controls (262.89,131.94) and (262.18,131.25) .. (262.18,130.39) -- cycle ;
\draw [fill={rgb, 255:red, 0; green, 0; blue, 0 }  ,fill opacity=1 ]   (442.81,70.21) -- (442.56,130.17) ;
\draw  [color={rgb, 255:red, 0; green, 0; blue, 0 }  ,draw opacity=1 ][fill={rgb, 255:red, 0; green, 0; blue, 0 }  ,fill opacity=1 ][line width=2.25]  (441.29,70.39) .. controls (441.29,69.53) and (442,68.84) .. (442.89,68.84) .. controls (443.77,68.84) and (444.49,69.53) .. (444.49,70.39) .. controls (444.49,71.25) and (443.77,71.95) .. (442.89,71.95) .. controls (442,71.95) and (441.29,71.25) .. (441.29,70.39) -- cycle ;
\draw  [color={rgb, 255:red, 0; green, 0; blue, 0 }  ,draw opacity=1 ][fill={rgb, 255:red, 0; green, 0; blue, 0 }  ,fill opacity=1 ][line width=2.25]  (440.95,130.17) .. controls (440.95,129.31) and (441.67,128.62) .. (442.56,128.62) .. controls (443.44,128.62) and (444.16,129.31) .. (444.16,130.17) .. controls (444.16,131.03) and (443.44,131.73) .. (442.56,131.73) .. controls (441.67,131.73) and (440.95,131.03) .. (440.95,130.17) -- cycle ;
\draw [fill={rgb, 255:red, 0; green, 0; blue, 0 }  ,fill opacity=1 ]   (442.56,99.67) -- (502.22,90) ;
\draw [fill={rgb, 255:red, 0; green, 0; blue, 0 }  ,fill opacity=1 ]   (442.56,99.67) -- (501.89,110) ;
\draw  [color={rgb, 255:red, 0; green, 0; blue, 0 }  ,draw opacity=1 ][fill={rgb, 255:red, 0; green, 0; blue, 0 }  ,fill opacity=1 ][line width=2.25]  (440.95,99.67) .. controls (440.95,98.81) and (441.67,98.11) .. (442.56,98.11) .. controls (443.44,98.11) and (444.16,98.81) .. (444.16,99.67) .. controls (444.16,100.53) and (443.44,101.22) .. (442.56,101.22) .. controls (441.67,101.22) and (440.95,100.53) .. (440.95,99.67) -- cycle ;
\draw [fill={rgb, 255:red, 0; green, 0; blue, 0 }  ,fill opacity=1 ]   (442.56,130.07) -- (502.22,120.4) ;
\draw [fill={rgb, 255:red, 0; green, 0; blue, 0 }  ,fill opacity=1 ]   (442.56,130.07) -- (501.89,140.4) ;
\draw [fill={rgb, 255:red, 0; green, 0; blue, 0 }  ,fill opacity=1 ]   (443.36,70.07) -- (503.02,60.4) ;
\draw [fill={rgb, 255:red, 0; green, 0; blue, 0 }  ,fill opacity=1 ]   (443.36,70.07) -- (502.69,80.4) ;

\draw (145,96) node [anchor=north west][inner sep=0.75pt]  [font=\scriptsize]  {$d$};
\draw (316,65) node [anchor=north west][inner sep=0.75pt]  [font=\scriptsize]  {$d-1$};
\draw (500,66) node [anchor=north west][inner sep=0.75pt]  [font=\scriptsize]  {$d-1$};
\draw (500,125) node [anchor=north west][inner sep=0.75pt]  [font=\scriptsize]  {$d-1$};
\draw (500,96) node [anchor=north west][inner sep=0.75pt]  [font=\scriptsize]  {$d-2$};
\draw (426.33,62.73) node [anchor=north west][inner sep=0.75pt]  [font=\scriptsize]  {$v_{1}$};
\draw (425.22,91.84) node [anchor=north west][inner sep=0.75pt]  [font=\scriptsize]  {$v_{2}$};
\draw (425.44,121.62) node [anchor=north west][inner sep=0.75pt]  [font=\scriptsize]  {$v_{3}$};
\draw (246.33,63.62) node [anchor=north west][inner sep=0.75pt]  [font=\scriptsize]  {$v_{1}$};
\draw (245.22,122.73) node [anchor=north west][inner sep=0.75pt]  [font=\scriptsize]  {$v_{2}$};
\draw (316,124) node [anchor=north west][inner sep=0.75pt]  [font=\scriptsize]  {$d-1$};
\draw (76.33,92.29) node [anchor=north west][inner sep=0.75pt]  [font=\scriptsize]  {$v_{1}$};
\draw (138,94) node [anchor=north west][inner sep=0.75pt]  [font=\large,rotate=-90.08] [align=left] {...};
\draw (310,64) node [anchor=north west][inner sep=0.75pt]  [font=\large,rotate=-90.08] [align=left] {...};
\draw (310,123) node [anchor=north west][inner sep=0.75pt]  [font=\large,rotate=-90.08] [align=left] {...};
\draw (494,124) node [anchor=north west][inner sep=0.75pt]  [font=\large,rotate=-90.08] [align=left] {...};
\draw (494,93) node [anchor=north west][inner sep=0.75pt]  [font=\large,rotate=-90.08] [align=left] {...};
\draw (494,64) node [anchor=north west][inner sep=0.75pt]  [font=\large,rotate=-90.08] [align=left] {...};

\draw (120,166) node [anchor=north west][inner sep=0.75pt]  [font=\scriptsize]  {$a)$};
\draw (290,166) node [anchor=north west][inner sep=0.75pt]  [font=\scriptsize]  {$b)$};
\draw (465,166) node [anchor=north west][inner sep=0.75pt]  [font=\scriptsize]  {$c)$};

\end{tikzpicture}

%% file: Figures/grados_cubo2.tikz
\tikzset{every picture/.style={line width=0.75pt}} 

\begin{tikzpicture}[x=0.75pt,y=0.75pt,yscale=-1.2,xscale=1.2]

\draw [fill={rgb, 255:red, 0; green, 0; blue, 0 }  ,fill opacity=1 ]   (80.81,59.21) -- (80.56,151.17) ;
\draw  [color={rgb, 255:red, 0; green, 0; blue, 0 }  ,draw opacity=1 ][fill={rgb, 255:red, 0; green, 0; blue, 0 }  ,fill opacity=1 ][line width=2.25]  (78.95,151.17) .. controls (78.95,150.31) and (79.67,149.62) .. (80.56,149.62) .. controls (81.44,149.62) and (82.16,150.31) .. (82.16,151.17) .. controls (82.16,152.03) and (81.44,152.73) .. (80.56,152.73) .. controls (79.67,152.73) and (78.95,152.03) .. (78.95,151.17) -- cycle ;
\draw [fill={rgb, 255:red, 0; green, 0; blue, 0 }  ,fill opacity=1 ]   (80.56,120.67) -- (140.22,111) ;
\draw [fill={rgb, 255:red, 0; green, 0; blue, 0 }  ,fill opacity=1 ]   (80.56,120.67) -- (139.89,131) ;
\draw  [color={rgb, 255:red, 0; green, 0; blue, 0 }  ,draw opacity=1 ][fill={rgb, 255:red, 0; green, 0; blue, 0 }  ,fill opacity=1 ][line width=2.25]  (78.95,120.67) .. controls (78.95,119.81) and (79.67,119.11) .. (80.56,119.11) .. controls (81.44,119.11) and (82.16,119.81) .. (82.16,120.67) .. controls (82.16,121.53) and (81.44,122.22) .. (80.56,122.22) .. controls (79.67,122.22) and (78.95,121.53) .. (78.95,120.67) -- cycle ;
\draw [fill={rgb, 255:red, 0; green, 0; blue, 0 }  ,fill opacity=1 ]   (80.56,151.07) -- (140.22,141.4) ;
\draw [fill={rgb, 255:red, 0; green, 0; blue, 0 }  ,fill opacity=1 ]   (80.56,151.07) -- (139.89,161.4) ;
\draw [fill={rgb, 255:red, 0; green, 0; blue, 0 }  ,fill opacity=1 ]   (81.36,91.07) -- (141.02,81.4) ;
\draw [fill={rgb, 255:red, 0; green, 0; blue, 0 }  ,fill opacity=1 ]   (81.36,91.07) -- (140.69,101.4) ;
\draw [fill={rgb, 255:red, 0; green, 0; blue, 0 }  ,fill opacity=1 ]   (80.36,59.57) -- (140.02,49.9) ;
\draw [fill={rgb, 255:red, 0; green, 0; blue, 0 }  ,fill opacity=1 ]   (80.36,59.57) -- (139.69,69.9) ;
\draw  [color={rgb, 255:red, 0; green, 0; blue, 0 }  ,draw opacity=1 ][fill={rgb, 255:red, 0; green, 0; blue, 0 }  ,fill opacity=1 ][line width=2.25]  (79.36,59.57) .. controls (79.36,58.71) and (80.07,58.01) .. (80.96,58.01) .. controls (81.84,58.01) and (82.56,58.71) .. (82.56,59.57) .. controls (82.56,60.43) and (81.84,61.12) .. (80.96,61.12) .. controls (80.07,61.12) and (79.36,60.43) .. (79.36,59.57) -- cycle ;
\draw  [color={rgb, 255:red, 0; green, 0; blue, 0 }  ,draw opacity=1 ][fill={rgb, 255:red, 0; green, 0; blue, 0 }  ,fill opacity=1 ][line width=2.25]  (79.75,91.07) .. controls (79.75,90.21) and (80.47,89.51) .. (81.36,89.51) .. controls (82.24,89.51) and (82.96,90.21) .. (82.96,91.07) .. controls (82.96,91.93) and (82.24,92.62) .. (81.36,92.62) .. controls (80.47,92.62) and (79.75,91.93) .. (79.75,91.07) -- cycle ;
\draw  [color={rgb, 255:red, 0; green, 0; blue, 0 }  ,draw opacity=1 ][fill={rgb, 255:red, 0; green, 0; blue, 0 }  ,fill opacity=1 ][line width=2.25]  (283.67,119.81) .. controls (283.67,118.95) and (284.38,118.26) .. (285.27,118.26) .. controls (286.15,118.26) and (286.87,118.95) .. (286.87,119.81) .. controls (286.87,120.67) and (286.15,121.37) .. (285.27,121.37) .. controls (284.38,121.37) and (283.67,120.67) .. (283.67,119.81) -- cycle ;
\draw [fill={rgb, 255:red, 0; green, 0; blue, 0 }  ,fill opacity=1 ]   (285.92,94.79) -- (345.59,85.12) ;
\draw [fill={rgb, 255:red, 0; green, 0; blue, 0 }  ,fill opacity=1 ]   (285.92,94.79) -- (345.25,105.12) ;
\draw  [color={rgb, 255:red, 0; green, 0; blue, 0 }  ,draw opacity=1 ][fill={rgb, 255:red, 0; green, 0; blue, 0 }  ,fill opacity=1 ][line width=2.25]  (284.32,94.79) .. controls (284.32,93.93) and (285.03,93.23) .. (285.92,93.23) .. controls (286.8,93.23) and (287.52,93.93) .. (287.52,94.79) .. controls (287.52,95.65) and (286.8,96.34) .. (285.92,96.34) .. controls (285.03,96.34) and (284.32,95.65) .. (284.32,94.79) -- cycle ;
\draw [fill={rgb, 255:red, 0; green, 0; blue, 0 }  ,fill opacity=1 ]   (285.27,119.71) -- (344.94,110.04) ;
\draw [fill={rgb, 255:red, 0; green, 0; blue, 0 }  ,fill opacity=1 ]   (285.27,119.71) -- (344.6,130.04) ;
\draw [fill={rgb, 255:red, 0; green, 0; blue, 0 }  ,fill opacity=1 ]   (286.36,69.28) -- (346.02,59.61) ;
\draw [fill={rgb, 255:red, 0; green, 0; blue, 0 }  ,fill opacity=1 ]   (286.36,69.28) -- (345.69,79.61) ;
\draw  [color={rgb, 255:red, 0; green, 0; blue, 0 }  ,draw opacity=1 ][fill={rgb, 255:red, 0; green, 0; blue, 0 }  ,fill opacity=1 ][line width=2.25]  (284.75,69.28) .. controls (284.75,68.42) and (285.47,67.72) .. (286.36,67.72) .. controls (287.24,67.72) and (287.96,68.42) .. (287.96,69.28) .. controls (287.96,70.14) and (287.24,70.83) .. (286.36,70.83) .. controls (285.47,70.83) and (284.75,70.14) .. (284.75,69.28) -- cycle ;
\draw [fill={rgb, 255:red, 0; green, 0; blue, 0 }  ,fill opacity=1 ]   (225.8,110.29) -- (285.8,120) ;
\draw [fill={rgb, 255:red, 0; green, 0; blue, 0 }  ,fill opacity=1 ]   (225.8,110.29) -- (286.37,94.86) ;
\draw [fill={rgb, 255:red, 0; green, 0; blue, 0 }  ,fill opacity=1 ]   (225.8,110.29) -- (286.36,69.28) ;
\draw  [color={rgb, 255:red, 0; green, 0; blue, 0 }  ,draw opacity=1 ][fill={rgb, 255:red, 0; green, 0; blue, 0 }  ,fill opacity=1 ][line width=2.25]  (224.2,110.29) .. controls (224.2,109.43) and (224.92,108.73) .. (225.8,108.73) .. controls (226.68,108.73) and (227.4,109.43) .. (227.4,110.29) .. controls (227.4,111.14) and (226.68,111.84) .. (225.8,111.84) .. controls (224.92,111.84) and (224.2,111.14) .. (224.2,110.29) -- cycle ;
\draw [fill={rgb, 255:red, 0; green, 0; blue, 0 }  ,fill opacity=1 ]   (470.42,116.45) -- (530.09,106.79) ;
\draw [fill={rgb, 255:red, 0; green, 0; blue, 0 }  ,fill opacity=1 ]   (470.42,116.45) -- (529.75,126.79) ;
\draw [fill={rgb, 255:red, 0; green, 0; blue, 0 }  ,fill opacity=1 ]   (470.52,85.61) -- (530.19,75.94) ;
\draw [fill={rgb, 255:red, 0; green, 0; blue, 0 }  ,fill opacity=1 ]   (470.52,85.61) -- (529.86,95.94) ;
\draw [fill={rgb, 255:red, 0; green, 0; blue, 0 }  ,fill opacity=1 ]   (470.52,85.61) -- (470.42,116.45) ;
\draw [fill={rgb, 255:red, 0; green, 0; blue, 0 }  ,fill opacity=1 ]   (450.19,76.07) -- (450.19,125.74) ;
\draw [fill={rgb, 255:red, 0; green, 0; blue, 0 }  ,fill opacity=1 ]   (470.42,116.45) -- (450.19,125.74) ;
\draw [fill={rgb, 255:red, 0; green, 0; blue, 0 }  ,fill opacity=1 ]   (470.52,85.61) -- (450.19,76.07) ;
\draw [fill={rgb, 255:red, 0; green, 0; blue, 0 }  ,fill opacity=1 ]   (450.19,76.07) -- (500.24,40.69) ;
\draw [fill={rgb, 255:red, 0; green, 0; blue, 0 }  ,fill opacity=1 ]   (450.19,76.07) -- (500.67,60.31) ;
\draw [fill={rgb, 255:red, 0; green, 0; blue, 0 }  ,fill opacity=1 ]   (450.19,125.74) -- (500.24,141.69) ;
\draw [fill={rgb, 255:red, 0; green, 0; blue, 0 }  ,fill opacity=1 ]   (450.19,125.74) -- (499.9,161.36) ;
\draw  [color={rgb, 255:red, 0; green, 0; blue, 0 }  ,draw opacity=1 ][fill={rgb, 255:red, 0; green, 0; blue, 0 }  ,fill opacity=1 ][line width=2.25]  (448.59,125.74) .. controls (448.59,124.88) and (449.31,124.18) .. (450.19,124.18) .. controls (451.08,124.18) and (451.79,124.88) .. (451.79,125.74) .. controls (451.79,126.6) and (451.08,127.29) .. (450.19,127.29) .. controls (449.31,127.29) and (448.59,126.6) .. (448.59,125.74) -- cycle ;
\draw  [color={rgb, 255:red, 0; green, 0; blue, 0 }  ,draw opacity=1 ][fill={rgb, 255:red, 0; green, 0; blue, 0 }  ,fill opacity=1 ][line width=2.25]  (448.59,76.07) .. controls (448.59,75.21) and (449.31,74.52) .. (450.19,74.52) .. controls (451.08,74.52) and (451.79,75.21) .. (451.79,76.07) .. controls (451.79,76.93) and (451.08,77.63) .. (450.19,77.63) .. controls (449.31,77.63) and (448.59,76.93) .. (448.59,76.07) -- cycle ;
\draw  [color={rgb, 255:red, 0; green, 0; blue, 0 }  ,draw opacity=1 ][fill={rgb, 255:red, 0; green, 0; blue, 0 }  ,fill opacity=1 ][line width=2.25]  (468.92,85.61) .. controls (468.92,84.75) and (469.64,84.06) .. (470.52,84.06) .. controls (471.41,84.06) and (472.12,84.75) .. (472.12,85.61) .. controls (472.12,86.47) and (471.41,87.17) .. (470.52,87.17) .. controls (469.64,87.17) and (468.92,86.47) .. (468.92,85.61) -- cycle ;
\draw  [color={rgb, 255:red, 0; green, 0; blue, 0 }  ,draw opacity=1 ][fill={rgb, 255:red, 0; green, 0; blue, 0 }  ,fill opacity=1 ][line width=2.25]  (468.82,116.45) .. controls (468.82,115.59) and (469.53,114.9) .. (470.42,114.9) .. controls (471.3,114.9) and (472.02,115.59) .. (472.02,116.45) .. controls (472.02,117.31) and (471.3,118.01) .. (470.42,118.01) .. controls (469.53,118.01) and (468.82,117.31) .. (468.82,116.45) -- cycle ;
\draw [fill={rgb, 255:red, 0; green, 0; blue, 0 }  ,fill opacity=1 ]   (225.8,110.29) -- (286.09,144.29) ;

\draw (142,55) node [anchor=north west][inner sep=0.75pt]  [font=\scriptsize]  {$d-1$};
\draw (142,146) node [anchor=north west][inner sep=0.75pt]  [font=\scriptsize]  {$d-1$};
\draw (142,116) node [anchor=north west][inner sep=0.75pt]  [font=\scriptsize]  {$d-2$};
\draw (64,55.23) node [anchor=north west][inner sep=0.75pt]  [font=\scriptsize]  {$v_{1}$};
\draw (64,84.34) node [anchor=north west][inner sep=0.75pt]  [font=\scriptsize]  {$v_{2}$};
\draw (64,114.12) node [anchor=north west][inner sep=0.75pt]  [font=\scriptsize]  {$v_{3}$};
\draw (135,145) node [anchor=north west][inner sep=0.75pt]  [font=\large,rotate=-90.08] [align=left] {...};
\draw (135,114) node [anchor=north west][inner sep=0.75pt]  [font=\large,rotate=-90.08] [align=left] {...};
\draw (135,85) node [anchor=north west][inner sep=0.75pt]  [font=\large,rotate=-90.08] [align=left] {...};
\draw (64,144.62) node [anchor=north west][inner sep=0.75pt]  [font=\scriptsize]  {$v_{4}$};
\draw (135,53) node [anchor=north west][inner sep=0.75pt]  [font=\large,rotate=-90.08] [align=left] {...};
\draw (142,88) node [anchor=north west][inner sep=0.75pt]  [font=\scriptsize]  {$d-2$};
\draw (348,116) node [anchor=north west][inner sep=0.75pt]  [font=\scriptsize]  {$d-1$};
\draw (348,90.23) node [anchor=north west][inner sep=0.75pt]  [font=\scriptsize]  {$d-1$};
\draw (280,52.98) node [anchor=north west][inner sep=0.75pt]  [font=\scriptsize]  {$v_{2}$};
\draw (280,79.55) node [anchor=north west][inner sep=0.75pt]  [font=\scriptsize]  {$v_{3}$};
\draw (340,114) node [anchor=north west][inner sep=0.75pt]  [font=\large,rotate=-90.08] [align=left] {...};
\draw (340,89) node [anchor=north west][inner sep=0.75pt]  [font=\large,rotate=-90.08] [align=left] {...};
\draw (340,63) node [anchor=north west][inner sep=0.75pt]  [font=\large,rotate=-90.08] [align=left] {...};
\draw (279,106) node [anchor=north west][inner sep=0.75pt]  [font=\scriptsize]  {$v_{4}$};
\draw (348,64.02) node [anchor=north west][inner sep=0.75pt]  [font=\scriptsize]  {$d-1$};
\draw (208,102.41) node [anchor=north west][inner sep=0.75pt]  [font=\scriptsize]  {$v_{1}$};
\draw (290,128) node [anchor=north west][inner sep=0.75pt]  [font=\scriptsize]  {$d-3$};
\draw (505,145.76) node [anchor=north west][inner sep=0.75pt]  [font=\scriptsize]  {$d-2$};
\draw (522,111) node [anchor=north west][inner sep=0.75pt]  [font=\large,rotate=-90.08] [align=left] {...};
\draw (522,80) node [anchor=north west][inner sep=0.75pt]  [font=\large,rotate=-90.08] [align=left] {...};
\draw (507,43.21) node [anchor=north west][inner sep=0.75pt]  [font=\scriptsize]  {$d-2$};
\draw (496,48) node [anchor=north west][inner sep=0.75pt]  [font=\large,rotate=-90.08] [align=left] {...};
\draw (496,142) node [anchor=north west][inner sep=0.75pt]  [font=\large,rotate=-90.08] [align=left] {...};
\draw (434,68.7) node [anchor=north west][inner sep=0.75pt]  [font=\scriptsize]  {$v_{1}$};
\draw (454,83.13) node [anchor=north west][inner sep=0.75pt]  [font=\scriptsize]  {$v_{2}$};
\draw (454,105.5) node [anchor=north west][inner sep=0.75pt]  [font=\scriptsize]  {$v_{3}$};
\draw (434,118.84) node [anchor=north west][inner sep=0.75pt]  [font=\scriptsize]  {$v_{4}$};
\draw (527,80.23) node [anchor=north west][inner sep=0.75pt]  [font=\scriptsize]  {$d-2$};
\draw (527,110.37) node [anchor=north west][inner sep=0.75pt]  [font=\scriptsize]  {$d-2$};
\draw (284,124) node [anchor=north west][inner sep=0.75pt]  [font=\large,rotate=-90.08] [align=left] {...};

\draw (120,178) node [anchor=north west][inner sep=0.75pt]  [font=\scriptsize]  {$a)$};
\draw (290,178) node [anchor=north west][inner sep=0.75pt]  [font=\scriptsize]  {$b)$};
\draw (465,178) node [anchor=north west][inner sep=0.75pt]  [font=\scriptsize]  {$c)$};
\end{tikzpicture}